\documentclass[12pt]{amsart}

\usepackage[hmargin=0.8in,height=8.6in]{geometry}
\usepackage{amssymb,amsthm, times}
\usepackage{delarray,verbatim}
\usepackage{ifpdf}
\ifpdf
\usepackage[pdftex]{graphicx}
 \else
\usepackage[dvips]{graphicx}
 \fi

\usepackage{bm}

\usepackage{ifpdf}
\usepackage{color}
\definecolor{webgreen}{rgb}{0,.5,0}
\definecolor{webbrown}{rgb}{.8,0,0}
\definecolor{emphcolor}{rgb}{0.95,0.95,0.95}

\usepackage{hyperref}
\hypersetup{%
          colorlinks=true,
          linkcolor=webbrown,
          filecolor=webbrown,
          citecolor=webgreen,
          breaklinks=true}
\ifpdf \hypersetup{pdftex,
            pdfstartview=FitH, 
            bookmarksopen=true,
            bookmarksnumbered=true
} \else \hypersetup{dvips} \fi

\newcommand{\ud}{{\rm d}}

\numberwithin{equation}{section}

\newtheorem{theorem}{Theorem}[section]
\newtheorem{proposition}{Proposition}[section]
\newtheorem{corollary}{Corollary}[section]
\newtheorem{remark}{Remark}[section]
\newtheorem{lemma}{Lemma}[section]

\newtheorem{claim}{Claim}[section]
\numberwithin{remark}{section} \numberwithin{proposition}{section}
\numberwithin{corollary}{section}
\newcommand {\R}{\mathbb{R}}

\newcommand {\p}{\mathbb{P}}

\newcommand {\E}{\mathbb{E}}

\newcommand{\diff}{{\rm d}}

\newcommand{\lev}{L\'{e}vy }

\newcommand{\e}{\mathbb{E}}

\title[Refracted-reflected Spectrally Negative L\'evy processes]{On the refracted-reflected Spectrally Negative L\'evy processes}
\thanks{This version: \today. }
\author[J. L. P\'erez]{Jos\'e-Luis P\'erez}
\address[J. L. P\'erez]{Department of Probability and Statistics, Centro de Investigaci\'on en Matem\'aticas A.C. Calle Jalisco s/n. C.P. 36240, Guanajuato, Mexico.}
\email{jluis.garmendia@cimat.mx }
\author[K. Yamazaki]{Kazutoshi Yamazaki}
\address[K. Yamazaki]{Department of Mathematics,
Faculty of Engineering Science, Kansai University, 3-3-35 Yamate-cho, Suita-shi, Osaka 564-8680, Japan.}
\email{kyamazak@kansai-u.ac.jp}
\date{}
\begin{document}

\begin{abstract}
We study  a combination of the refracted and reflected \lev processes. Given a spectrally negative \lev process and two boundaries, it is reflected at the lower boundary while, 
whenever it is above the upper boundary, a linear drift at a constant rate is subtracted from the increments of the process. Using the scale functions, we compute the resolvent measure, the Laplace transform of the occupation times as well as other fluctuation identities that will be useful in applied probability including insurance, queues, and inventory management.

\noindent \small{\textbf{Key words:}  
 \lev processes; fluctuation theory; scale functions; insurance risk.  }\\
\end{abstract}

\maketitle

\section{Introduction}

In this paper, we study what we call the \emph{refracted-reflected} spectrally negative \lev process. Its dynamics can be understood as follows: given a spectrally negative \lev process and two boundaries, we reflect the path of the process at the lower boundary while, whenever it is above the upper boundary, a linear drift at a constant rate is subtracted from the increments of the process.

This process can be seen as a combination of a L\'evy process reflected at the lower boundary and the refracted \lev process, introduced by Kyprianou and Loeffen \cite{KL}, with the upper boundary being the refraction trigger level. 
The former is well-studied and is known to be expressed as the difference between the underlying \lev process and its running infimum process.  The latter moves like the original process below a fixed level while it behaves like a drift-changed process above it. Various fluctuation identities have been developed for reflected and refracted L\'evy processes via the use of scale functions. We refer the reader to \cite{APP2007,P2004,P2007} and \cite{KL, KPP} for a review on the study of the former and the latter, respectively.

Our study of the refracted-reflected L\'evy process is motivated by its potential applications in applied probability, as exemplified by insurance mathematics and queueing theory.

In insurance, the classical Cram\'er-Lundberg model uses a compound Poisson process with negative jumps as the surplus of an insurance company. Nevertheless, very recent studies motivated by insurance risk have seen preference in working with a general spectrally negative \lev process, partly thanks to the development of the fluctuation theory of L\'evy processes and scale functions. See for example \cite{APP2007,F1998,HPSV2004a,HPSV2004b,KKM2004,KK2006,KP2007,R2014,RZ2007,SV2007}. 
In particular, in \emph{de Finetti's optimal dividend problem}, an insurance company aims to maximize the expected net present value (NPV) of the total
	dividends paid until ruin.  It is shown, under a certain condition (see \cite{Loeffen}), that the barrier strategy is optimal and the resulting controlled surplus process becomes the \lev process reflected at an upper boundary; see, among others, \cite{APP2007}  and \cite{Loeffen}.

To this classical setting, there are two existing extensions. First, in the \emph{bail-out} model, it is assumed that the capital must be injected to prevent the surplus process from going below zero.  In this setting, Avram et al. \cite{APP2007} showed that it is optimal to reflect at zero and at some upper boundary, with the optimally controlled process being the doubly reflected \lev process of \cite{Pistorius_2003}. Second, under the restriction that the rate at which the dividends are paid is bounded and, instead,  absolutely continuous, Kyprianou et al.\ \cite{KLP} showed that, if the L\'evy measure has a completely monotone density,  it is optimal to pay dividends at the maximal rate as long as the surplus is above some fixed level.
The optimally controlled process becomes then the refracted \lev process.  
Naturally, it is of great interest to think of a joint problem with both capital injection and an absolutely continuous assumption; a refracted-reflected \lev process is an obvious candidate for the optimally controlled process. 

In queues and inventory management, a \lev process is also a main tool in modeling. As an important characteristic of queues and inventories, they have a lower bound at zero, and hence it is common to model them as those reflected at zero.  As a more realistic model, \emph{queues with abandonments} take into account the well-studied phenomenon that a customer is not patient enough to line up if the queue is too long; consequently the rate of  increments of a queue decreases when the current level is high. On the other hand in inventory management, the rate of replenishment must be decreased when the  inventory level is high; this is due to the fact that there is a limited capacity of inventory and it is necessary to reduce future unsalable stock.
  In these settings, refracted-reflected L\'evy processes are again a natural choice to model these processes.

In this paper, we construct and study the refracted-reflected process when the underlying process is a spectrally negative \lev process. Our objective is to compute several fluctuation identities including the resolvent measure, the expected NPV of capital injection  (in the insurance context described above) and the Laplace transform of the occupation time above and below the level of refraction.

Given that this type of processes constitute an extension of the refracted case, we apply and adopt several techniques used in \cite{KL}.  We shall first consider the bounded variation case, and then by an approximation scheme we study the unbounded variation case. In order to avoid complicated expressions in terms of the \lev measure, we obtain several formulae to simplify the expressions of the obtained fluctuation identities.

Our challenges are mainly caused by the negative jumps.  There are essentially three regions that determine the movement of the process: (1) the refraction region above the upper boundary, (2) the reflection region below the lower boundary, and (3) the waiting region between these. Due to negative jumps, the process can potentially jump from the refraction region to any of the other remaining regions.
We shall show, nevertheless, that the fluctuation identities can be obtained and can be simplified by the formulae discussed above.


The rest of the paper is organized as follows.  Section \ref{section_process}  gives a construction of the refracted-reflected process. Section  \ref{section_scale_functions} reviews the theory of scale functions and develops new simplifying formulae that will be useful in the sequel and are potentially beneficial in future work. 
Section \ref{section_resolvents} computes the resolvent measure and, as special cases, we obtain the one-sided exit identity as well as the expected NPV of dividends in the insurance context.  In Section \ref{section_capital_injection},  we obtain the expected NPV of capital injection.
 Finally, in Section \ref{section_occupation_time}, we study the occupation time of the process above and below the refraction level.

\section{Refracted-reflected spectrally one-sided \lev processes} \label{section_process}
\subsection{Spectrally negative \lev processes and their reflected/refracted processes}
Let $X=(X_t; t\geq 0)$ be a L\'evy process defined on a  probability space $(\Omega, \mathcal{F}, \p)$.  For $x\in \R$, we denote by $\p_x$ the law of $X$ when it starts at $x$ and write for convenience  $\p$ in place of $\p_0$. Accordingly, we shall write $\e_x$ and $\e$ for the associated expectation operators. 
We let $(\mathcal{F}_t)_{t \geq 0}$ be the filtration generated by $X$. In this paper, we shall assume throughout that $X$ is \textit{spectrally negative},   meaning here that it has no positive jumps and that it is not the negative of a subordinator. This allows us to define the Laplace exponent $\psi(\theta):[0,\infty) \to \R$, i.e.
\[
\e\big[{\rm e}^{\theta X_t}\big]=:{\rm e}^{\psi(\theta)t}, \qquad t, \theta\ge 0,
\]
given by the \emph{L\'evy-Khintchine formula}
\begin{equation}
\psi(\theta):=\gamma\theta+\frac{\sigma^2}{2}\theta^2+\int_{(-\infty,0)}\big({\rm e}^{\theta x}-1-\theta x\mathbf{1}_{\{x>-1\}}\big)\Pi(\ud x), \quad \theta \geq 0,\notag
\end{equation}
where $\gamma\in \R$, $\sigma\ge 0$, and $\Pi$ is a measure on $(-\infty,0)$ called the L\'evy measure of $X$ that satisfies
\[
\int_{(-\infty,0)}(1\land x^2)\Pi(\ud x)<\infty.
\]

It is well-known that $X$ has paths of bounded variation if and only if $\sigma=0$ and $\int_{(-1, 0)} x\Pi(\mathrm{d}x)$ is finite. In this case $X$ can be written as
\begin{equation}
X_t=ct-S_t, \,\,\qquad t\geq 0,\notag
\end{equation}
where 
\begin{align}
c:=\gamma-\int_{(-1,0)} x\Pi(\mathrm{d}x) \label{def_drift_finite_var}
\end{align}
 and $(S_t; t\geq0)$ is a driftless subordinator. Note that  necessarily $c>0$, since we have ruled out the case that $X$ has monotone paths; its Laplace exponent is given by
\begin{equation*}
\psi(\theta) = c \theta+\int_{(-\infty,0)}\big( {\rm e}^{\theta x}-1\big)\Pi(\ud x), \quad \theta \geq 0.
\end{equation*}

The \emph{\lev process reflected at the lower boundary} $0$ is a strong Markov process written concisely by
\begin{align}
U_t := X_t+\sup_{0 \leq s\leq t}(-X_s)\vee0,\qquad t\geq0. \label{reflected_levy}
\end{align}
The supremum term pushes the process upward whenever it attempts to down-cross the level $0$; as a result the process only takes values on $[0, \infty)$. 

The \emph{refracted \lev process} is a variant of the reflected \lev process, first introduced by Kyprianou and Loeffen \cite{KL}.
Informally speaking, a linear drift at rate $\delta>0$ is subtracted from the increments of  the underlying \lev process $X$ whenever it exceeds a pre-specified positive level $b>0$. More formally, it is the unique strong solution to the stochastic differential equation given by
\begin{equation}\label{SDE}
\ud A_t=\ud X_t-\delta\mathbf{1}_{\{A_t>b\}}\ud t, \qquad t\geq 0.
\end{equation}
In deriving the fluctuation identities, it is important that the \emph{drift-changed process}
\begin{align}
Y_t := X_t - \delta t, \quad t \geq 0, \label{def_Y}
\end{align}
is again a spectrally negative \lev process that is not the negative of a subordinator.  Hence,  in \cite{KL} and \cite{KLP}, the standing assumption (with $c$ defined as in \eqref{def_drift_finite_var}) 
\begin{equation}
\mathrm{({\bf H})}\qquad\delta< c,\qquad\text{if $X$ has paths of bounded variation}, \notag
\end{equation}
is imposed.

The reader is referred to Bertoin \cite{B} and Kyprianou \cite{K} for a complete introduction to the theory of L\'evy processes and their reflected processes.  For refracted \lev processes, see \cite{KL}, \cite{KLP}, \cite{KPP}, and \cite{R2014}.


\subsection{Refracted-reflected \lev processes}\label{RRLPD}
For the rest of the paper, we fix $b > 0$ and $\delta \geq 0$ such that condition $\mathrm{({\bf H})}$ above holds.
We define the \emph{refracted-reflected \lev process} as follows.
While the process is above $b$ a linear drift at rate $\delta$ is subtracted from the increments of process. On the other hand, when it attempts to down-cross $0$, it is pushed upward so that it will not go below $0$.



The process can be formally constructed by the recursive algorithm given below:

\begin{center}
\line(1,0){300}
\end{center}
 \textbf{Construction of the refracted-reflected \lev process $V$ under $\p_x$}
\begin{description}
\item[Step 0] Set $V_{0-}=x$.  If $x \geq 0$, then set $\underline{\tau} := 0$ and go to \textbf{Step 1}.  Otherwise, set $\overline{\tau} := 0$ and go to \textbf{Step 2}.
\item[Step 1] Let $\{ \widetilde{A}_t; t \geq \underline{\tau} \}$ be the refracted \lev process (with refraction level $b$) that starts at the time $\underline{\tau}$ at the level $x$, and $\overline{\tau} := \inf \{ t > \underline{\tau}: \widetilde{A}_t < 0\}$.
Set $V_t=\widetilde{A}_t$ for all $\underline{\tau} \leq t < \overline{\tau}$. Then go to \textbf{Step 2}.
\item[Step 2] Let $\{ \widetilde{U}_t; t \geq \overline{\tau}\}$ be the \lev process reflected at the lower boundary $0$ that starts at time $\overline{\tau}$ at $0$, and $\underline{\tau} := \inf \{ t > \overline{\tau}: \widetilde{U}_t > b\}$. Set $V_t=\widetilde{U}_t$ for all $\overline{\tau} \leq t < \underline{\tau}$ and $x = b$. Then go to  \textbf{Step 1}.\end{description}
\begin{center}
\line(1,0){300}
\end{center}

Let us define the nondecreasing and right-continuous processes $R_t$ and $L_t$ that represent, respectively, the cumulative amounts of modification up to $t$ that pushes the process upward when it attempts to go below $0$ and that pushes downward  when it is above $b$.  Then, we have a decomposition
\begin{align*}
V_t = X_t + R_t - L_t, \quad t \geq 0,
\end{align*}
where we can write 
\begin{align*}
L_t = \delta \int_0^t 1_{\{ V_s > b \}} \ud s, \quad t \geq 0.
\end{align*}
In particular, \emph{for the case of bounded variation,} 
\begin{align*}
R_t = \sum_{t\geq0: V_{t-} + \Delta X_t < 0} |V_{t-} + \Delta X_t | \quad t \geq 0.
\end{align*}
Here and for the rest of the paper, we define $\Delta \xi_t := \xi_t - \xi_{t-}$, $t \geq 0$, for any right-continuous process $\xi$.


\subsection{Applications} \label{subsection_dividends}
 As discussed in the introduction, the process $V$ can model the surplus of a dividend-paying company under the bail-out setting: they receive an injection of capital so as to prevent the process from going below the default level $0$, while they also pay dividends at a fixed rate $\delta > 0$ whenever the surplus is above the level $b$.  Under this scenario, it is clear from the construction above that $L$ is the cumulative amount of dividend payments whereas $R$ is that of injected capital.

As we have reviewed in the introduction, $V$ can also model queues with abandonments as well as inventories. In the former, $L$ models the number of customers who decide not to line up because it is too long.  In the latter, $R$ models the lost sales due to the lack of stock.


Motivated by these applications,  it is thus of great interest to compute the expected NPV's:
 \begin{align} \label{net_present_value}
 \mathbb{E}_x \Big(\int_0^\infty e^{-qt}\ud L_t \Big)= \delta \mathbb{E}_x  \Big(\int_0^\infty e^{-qt}1_{\{V_s>b\}}\ud s \Big) \quad \text{ and } \quad \mathbb{E}_x \Big(\int_{[0, \infty)} e^{-qt}\ud R_t \Big).
 \end{align}
 We compute these values as corollaries to the fluctuation identities of $V$; see Sections \ref{section_resolvents} and \ref{section_capital_injection}.

In addition, it is worth studying further the length of the time $\{ s\geq 0: V_s > b \}$.  In particular, in the queues with abandonments, it is important to compute the distribution of the occupation times $\int 1_{\{ V_s > b \}} \diff s$ and  $\int 1_{\{ V_s < b \}} \diff s$ where the former is the length of the busy period where customers may decide not to line up and the latter is that of the non-busy period; we obtain their Laplace transforms in Section \ref{section_occupation_time}.

\subsection{Approximation of the case of unbounded variation} As has been done in \cite{KL}, our approach first obtains identities for the case in which $X$ is of bounded variation and then extends the results to the unbounded variation case.  In order to carry this out, we need to show that the refracted-reflected \lev process $V$ associated with the \lev process $X$ of unbounded variation can be approximated by those of bounded variation. 

Following Definition 11 of \cite{KL}, given a stochastic process $(\xi_s; s \geq 0)$, a sequence of processes $\{(\xi_s^{(n)})_{s \geq 0};n\geq1\}$ is strongly approximating for $\xi$, if $\lim_{n \uparrow \infty}\sup_{0 \leq s \leq t} |\xi_s - \xi^{(n)}_{s}| =0$ for any $t > 0$ a.s.  As is well-known (see Definition 11 of \cite{KL} and page 210 of \cite{B}), for any spectrally negative \lev process $X$, there exists a strongly approximating sequence $X^{(n)}$ with paths of bounded variation.  We shall obtain a similar result for the refracted-reflected \lev process.

\begin{proposition} \label{prop_approximation}
Suppose $X$ is of unbounded variation and $(X^{(n)}; n \geq 1)$ is a strongly approximating sequence for $X$. In addition, define $V$ and  $V^{(n)}$ as the refracted-reflected processes associated with  $X$ and $X^{(n)}$, respectively. Then $V^{(n)}$ is a strongly approximating sequence of $V$. 
\end{proposition}
\begin{proof}
See Appendix \ref{proof_prop_approximation}.
\end{proof}

\section{Review of scale functions and some new identities}  \label{section_scale_functions}In this section, we review the scale function of spectrally negative \lev processes and develop some new identities that will be used to simplify the expression of the fluctuation identities in our later arguments.

Fix $q \geq 0$. Following the same notations as in \cite{KL}, we use $W^{(q)}$ and $\mathbb{W}^{(q)}$ for the scale functions of $X$ and $Y$ (see \eqref{def_Y} for the latter), respectively.  Namely,  these are the mappings from $\R$ to $[0, \infty)$ that take value zero on the negative half-line, while on the positive half-line they are strictly increasing functions that are defined by their Laplace transforms:
\begin{align} \label{scale_function_laplace}
\begin{split}
\int_0^\infty  \mathrm{e}^{-\theta x} W^{(q)}(x) \diff x &= \frac 1 {\psi(\theta)-q}, \quad \theta > \Phi(q), \\
\int_0^\infty  \mathrm{e}^{-\theta x} \mathbb{W}^{(q)}(x) \diff x &= \frac 1 {\psi_Y(\theta) -q}, \quad \theta > \varphi(q),
\end{split}
\end{align}
where $\psi_Y(\theta) := \psi(\theta) - \delta \theta$, $\theta \geq 0$, is the Laplace exponent for $Y$ and
\begin{align}
\begin{split}
\Phi(q) := \sup \{ \lambda \geq 0: \psi(\lambda) = q\} \quad \textrm{and} \quad \varphi(q) := \sup \{ \lambda \geq 0: \psi_Y(\lambda) = q\}. \notag
\end{split}
\end{align}
In particular, when $q=0$, we shall drop the superscript.
By the strict  convexity of $\psi$, we derive the inequality $\varphi(q) > \Phi(q) > 0$ for $q > 0$ and   $\varphi(0) \geq \Phi(0) \geq 0$.

We also define, for $x \in \R$,
\begin{align*}
\overline{W}^{(q)}(x) &:=  \int_0^x W^{(q)}(y) \diff y, \\
Z^{(q)}(x) &:= 1 + q \overline{W}^{(q)}(x),  \\
\overline{Z}^{(q)}(x) &:= \int_0^x Z^{(q)} (z) \diff z = x + q \int_0^x \int_0^z W^{(q)} (w) \diff w \diff z.
\end{align*}
Noting that $W^{(q)}(x) = 0$ for $-\infty < x < 0$, we have
\begin{align}
\overline{W}^{(q)}(x) = 0, \quad Z^{(q)}(x) = 1  \quad \textrm{and} \quad \overline{Z}^{(q)}(x) = x, \quad x \leq 0.  \label{z_below_zero}
\end{align}
In addition, we define $\overline{\mathbb{W}}^{(q)}$, $\mathbb{Z}^{(q)}$ and $\overline{\mathbb{Z}}^{(q)}$ analogously for $Y$.  The scale functions of $X$ and $Y$ are related, for $x \in \R$ and 
$p, q \geq 0$,  by the following equalities
%
\begin{align}\label{RLqp}
&\int_0^x\mathbb{W}^{(p)}(x-y)\big[\delta W^{(q)}(y)-(q-p)\overline{W}^{(q)}(y)\big]\ud y=\overline{\mathbb{W}}^{(p)}(x)-\overline{W}^{(q)}(x), \quad \\
&\int_0^x\mathbb{W}^{(p)}(x-y)\big[\delta Z^{(q)}(y)-(q-p)\overline{Z}^{(q)}(y)\big]\ud y=\overline{\mathbb{Z}}^{(p)}(x)-\overline{Z}^{(q)}(x)+\delta\overline{\mathbb{W}}^{(p)}(x),\label{RLqp2}
\end{align}
which can be proven by showing that the Laplace transforms on both sides are equal.

Regarding their asymptotic values as $x \downarrow 0$ we have, as in Lemma 3.1 of \cite{KKR},
\begin{align*}
\begin{split}
W^{(q)} (0) &= \left\{ \begin{array}{ll} 0 & \textrm{if $X$ is of unbounded
variation,} \\ c^{-1} & \textrm{if $X$ is of bounded variation,}
\end{array} \right.  \\ \mathbb{W}^{(q)} (0) &= \left\{ \begin{array}{ll} 0 & \textrm{if $Y$ is of unbounded
variation,} \\ (c-\delta)^{-1} & \textrm{if $Y$ is of bounded variation,}
\end{array} \right. 
\end{split}
\end{align*}
and, as in Lemma 3.3 of \cite{KKR}, 
\begin{align}
\begin{split}
e^{-\Phi(q) x}W^{(q)} (x) \nearrow \psi'(\Phi(q))^{-1} \quad \textrm{and} \quad e^{-\varphi(q) x}\mathbb{W}^{(q)} (x) \nearrow \psi_Y'(\varphi(q))^{-1}, \quad \textrm{as } x \rightarrow \infty,
\end{split}
\label{W_q_limit}
\end{align}
where in the case $\psi'(0+) = 0$ or $\psi'_Y(0+) = 0$, the right hand side, when $q=0$,  is understood to be infinity.

The increments of the refracted-reflected \lev process can be decomposed into those of $Y$ and $U$ that are defined in \eqref{def_Y} and \eqref{reflected_levy}, respectively.  Here, we summarize a few known identities of these processes in terms of the scale function that will be used later in the paper.

For the drift-changed process $Y$, let us define the first down- and up-crossing times, respectively, by
\begin{align}
\label{first_passage_time}
\tau_a^- := \inf \left\{ t > 0: Y_t < a \right\} \quad \textrm{and} \quad \tau_a^+ := \inf \left\{ t > 0: Y_t >  a \right\}, \quad a \in \R;
\end{align}
here and throughout, let $\inf \varnothing = \infty$.
Then, for any $a > b$ and $x \leq a$,
\begin{align}
\begin{split}
\E_x \left( e^{-q \tau_a^+} 1_{\left\{ \tau_a^+ < \tau_b^- \right\}}\right) &= \frac {\mathbb{W}^{(q)}(x-b)}  {\mathbb{W}^{(q)}(a-b)}, \\
 \E_x \left( e^{-q \tau_b^-} 1_{\left\{ \tau_a^+ > \tau_b^- \right\}}\right) &= \mathbb{Z}^{(q)}(x-b) -  \mathbb{Z}^{(q)}(a-b) \frac {\mathbb{W}^{(q)}(x-b)}  {\mathbb{W}^{(q)}(a-b)}.
\end{split}
 \label{laplace_in_terms_of_z}
\end{align}
In addition, the \emph{$q$-resolvent measure} is known to have a density written as 
\begin{align} \label{resolvent_density}
\E_x \Big( \int_0^{\tau_{b}^- \wedge \tau^+_a} e^{-qt} 1_{\left\{ Y_t \in \diff y \right\}} \diff t\Big) &= \Big[ \frac {\mathbb{W}^{(q)}(x-b) \mathbb{W}^{(q)} (a-y)} {\mathbb{W}^{(q)}(a-b)} -\mathbb{W}^{(q)} (x-y) \Big] \diff y,  \quad b < x < a;
\end{align}
see Theorem 8.7 of \cite{K}.

It is known that a spectrally negative \lev process creeps downwards (i.e.\ $\p_x \{ Y_{\tau_b^-}= b, \tau_b^- < \infty \} > 0$ for $x > b$) if and only if $\sigma > 0$ (see Exercise 7.6 of \cite{K}). Hence, \emph{for the case of bounded variation}, the above identity \eqref{resolvent_density} together with the compensation formula (see Theorem 4.4 of \cite{K}) gives, for any $b < a$ and a nonnegative measurable function $l$,
\begin{multline} \label{undershoot_expectation}
\E_x\left(e^{-q\tau_b^-}l(Y_{\tau_b^-})1_{\{\tau_b^-<\tau_a^+\}}\right)\\
=\int_0^{a-b}\int_{(-\infty,-y)} l(y+u+b)\left\{\frac{\mathbb{W}^{(q)}(x-b) \mathbb{W}^{(q)}(a-b-y)}{ \mathbb{W}^{(q)}(a-b)}-\mathbb{W}^{(q)}(x-b-y)\right\}\Pi(\ud u)\ud y, \quad x \in \R. 
\end{multline}

Similar identities can also be computed for the reflected \lev process $U$ as in  \eqref{reflected_levy}. Let 
\begin{align}
\kappa^+_b:=\inf\{t>0:U_t>b\}. \label{def_kappa_time}
\end{align}
By Theorem 1 (i) of \cite{P2004}, for any Borel set $B$,
\begin{align}\label{rsr}
\e_x\bigg(\int_0^{\kappa_b^+}e^{-qt}1_{\{ U_t \in B \}}\ud t \bigg)=\frac{Z^{(q)}(x)}{Z^{(q)}(b)}\rho^{(q)}(b;B)-\rho^{(q)}(x;B), \quad x \leq b,
\end{align} 
where
\[
\rho^{(q)}(x;B) :=\int_0^xW^{(q)}(x-y)1_{\{y\in B\}}\ud y.
\]
In particular,
\begin{align} \label{upcrossing_time_reflected}
\e_x\big(e^{-q\kappa_b^+}\big)=\frac{Z^{(q)}(x)}{Z^{(q)}(b)}, \quad x \leq b.
\end{align}
In addition, if we define, for $t \geq 0$, $\tilde{R}_t := \sup_{s\leq t}(-X_s)\vee0$ so that $U_t = X_t + \tilde{R}_t$, then
\begin{align}
\mathbb{E}_x\Big(\int_{[0,\kappa_b^+]}e^{-qt}\ud \tilde{R}_t\Big)&=- \Big( \overline{Z}^{(q)}(x) + \frac {\psi'(0+)} q\Big) +\Big( \overline{Z}^{(q)}(b) + \frac {\psi'(0+)} q \Big) \frac{Z^{(q)}(x)} {Z^{(q)}(b)}, \quad x \leq b; \label{capital_injection_identity_SN}
\end{align}
see page 167 in the proof of Theorem 1 of \cite{APP2007}. It is noted that \eqref{upcrossing_time_reflected} and \eqref{capital_injection_identity_SN} hold even for $x < 0$ by \eqref{z_below_zero}.


\begin{remark} \label{remark_strongly_approximating}Suppose $(X^{(n)}; n \geq 1)$ is a strongly approximating sequence for $X$ and $(W^{(q)}_n; n \geq 1)$ and $(\mathbb{W}^{(q)}_n; n \geq 1)$ are the corresponding scale functions of $X^{(n)}$ and  $Y^{(n)} :=( X^{(n)}_t - \delta t, t \geq 0)$ respectively. Then, by the continuity theorem, as is discussed in Lemma 20 of \cite{KL}, $W^{(q)}_n(x)$ (resp.\ $\mathbb{W}^{(q)}_n(x)$) converges to $W^{(q)}(x)$ (resp.\ $\mathbb{W}^{(q)}(x)$) for every $x \in \R$. In addition, as recently shown in Remark 3.2 of \cite{PY_astin}, (if $X$ is of unbounded variation) $W^{(q)\prime}_n(x+)$ (resp.\ $\mathbb{W}^{(q)\prime}_n(x+)$) converges to $W^{(q)\prime}(x)$ (resp.\ $\mathbb{W}^{(q)\prime}(x)$) for all $x > 0$. 
This observation together with Proposition \ref{prop_approximation} is  used to extend the result from the case of bounded variation to that of unbounded variation.
\end{remark}

We conclude this section with some new identities that will simplify the expressions of our results and  will help us avoid the use of the \lev measure. In particular, the first item below is a generalization of (4.17) in \cite{KL}. 
\begin{lemma} \label{lemma_useful_identity}
Suppose $X$ is of bounded variation. For any $p,q\geq0$ and  $v \leq b \leq x$, we have 
\begin{itemize}
\item[(i)]
\begin{align} \label{useful_identity_W}
\begin{split}
\int_{0}^{\infty}&\int_{(-\infty,-y)}W^{(q)}(y+u+b-v)\mathbb{W}^{(p)}(x-b-y)\Pi(\ud u)\ud y \\
&=(c-\delta){W^{(q)}(b-v)} \mathbb{W}^{(p)}(x-b)-W^{(q)}(x-v)-\delta\int_b^x\mathbb{W}^{(p)}(x-y)W^{(q)\prime}(y-v)\ud y \\
&+(q-p)\int_b^x \mathbb{W}^{(p)} (x-z) W^{(q)} (z-v) \ud z,
\end{split}
\end{align}
\item[(ii)]
\begin{align} \label{lemma_useful_identity_c}
\begin{split}
\int_{0}^{\infty}\int_{(-\infty,-y)} &Z^{(q)}(y+u+b-v)\mathbb{W}^{(p)}(x-b-y)\Pi(\ud u)\ud y \\
&= (c-\delta) Z^{(q)}(b-v) \mathbb{W}^{(p)}(x-b)-Z^{(q)}(x-v)-(p-q)\overline{\mathbb{W}}^{(p)}(x-b) \\
&+q\int_b^x\mathbb{W}^{(p)}(x-y)\left((q-p)\overline{W}^{(q)}(y-v)-\delta W^{(q)}(y-v)\right)\ud y, \end{split}
\end{align}
\item[(iii)]
\begin{align} \label{lemma_useful_identity_Z_bar}
\begin{split}
\int_0^{\infty}&\int_{(-\infty,-y)}\overline{Z}^{(q)}(y+u+b-v)\mathbb{W}^{(p)}(x-b-y)\Pi(\ud u)\ud y \\
&=(c-\delta)\overline{Z}^{(q)}(b-v) \mathbb{W}^{(p)}(x-b)-\overline{Z}^{(q)}(x-v)-\delta\int_b^x\mathbb{W}^{(p)}(x-y)Z^{(q)}(y-v)\ud y \\
&+(q-p)\int_b^x\mathbb{W}^{(p)}(x-y)\overline{Z}^{(q)}(y-v) \ud y +\psi'(0+)\overline{\mathbb{W}}^{(p)}(x-b).
\end{split}
\end{align}
\end{itemize}
\end{lemma}
\begin{proof}
The formulae (i) and (ii) can be derived directly using \eqref{undershoot_expectation} and Lemma  1 of Renaud \cite{R2014}, which obtains \eqref{cc0_general} below and the same expression with $Z^{(p+q)}$ replaced with $W^{(p+q)}$. For the proof of (iii), see Appendix \ref{proof_lemma_useful_identit}. 
\end{proof}

These expressions are to be frequently used in later arguments.  In particular, 
by \eqref{undershoot_expectation} and (\ref{lemma_useful_identity_c}), as obtained in Lemma  1 of Renaud \cite{R2014}, for $x \geq b$, and $p,p+q \geq 0$,
\begin{align} \label{cc0_general}
\begin{split}
&\mathbb{E}_x\left(e^{-p\tau_b^-}Z^{(p+q)}(Y_{\tau_b^-})1_{\{\tau_b^-<\tau_a^+\}}\right) \\ &=
\int_0^{a-b}\int_{(-\infty,-y)} Z^{(p+q)}(y+u+b)\left\{\frac{\mathbb{W}^{(p)}(x-b)\mathbb{W}^{(p)}(a-b-y)}{\mathbb{W}^{(p)}(a-b)}-\mathbb{W}^{(p)}(x-b-y)\right\}\Pi(\ud u)\ud y\\
&=\mathcal{R}^{(p,q)}(x)-\mathcal{R}^{(p,q)}(a)\frac{\mathbb{W}^{(p)}(x-b)}{\mathbb{W}^{(p)}(a-b)},
\end{split}
\end{align}
where 
\begin{multline} \label{mathcal_R_def}
\mathcal{R}^{(p,q)}(x):=Z^{(p+q)}(x)-q\overline{\mathbb{W}}^{(p)}(x-b)\\-(p+q)\int_b^x\mathbb{W}^{(p)}(x-y)\left(q\overline{W}^{(p+q)}(y)-\delta W^{(p+q)}(y)\right)\ud y, \quad x \in \R, \; p, p+q \geq 0.
\end{multline}
As its special case, we have 
\begin{align}
\label{cc0}
\begin{split}
\mathbb{E}_x&\left(e^{-q\tau_b^-}Z^{(q)}(Y_{\tau_b^-})1_{\{\tau_b^-<\tau_a^+\}}\right)\\
&=
\int_0^{a-b}\int_{(-\infty,-y)} Z^{(q)}(y+u+b)\left\{\frac{\mathbb{W}^{(q)}(x-b)\mathbb{W}^{(q)}(a-b-y)}{\mathbb{W}^{(q)}(a-b)}-\mathbb{W}^{(q)}(x-b-y)\right\}\Pi(\ud u)\ud y\\ 
&=r^{(q)}(x) -\frac{\mathbb{W}^{(q)}(x-b)}{\mathbb{W}^{(q)}(a-b)} r^{(q)}(a)
\end{split}
\end{align}
where 
\begin{align*}
r^{(q)}(x)&:= \mathcal{R}^{(q,0)}(x) = Z^{(q)}(x)+q\delta\int_b^x\mathbb{W}^{(q)}(x-y)W^{(q)}(y)\ud y, \quad x \in \R, \; q \geq 0.
\end{align*}

Note that similar expressions to \eqref{cc0_general} and \eqref{cc0}  with $Z$ replaced with $W$ and $\overline{Z}$ can be computed using \eqref{useful_identity_W} and \eqref{lemma_useful_identity_Z_bar}, respectively.


\section{Resolvent Measures} \label{section_resolvents}



In this section, we study the resolvent measure and, as its byproducts, we also obtain the Laplace transform of the up-crossing time and the expected NPV of $L$ as defined in \eqref{net_present_value}.  
Let us define the following set of stopping times 
\[
T^+_a:=\inf\{t>0:V_t>a\} \quad \textrm{and} \quad T ^-_a:=\inf\{t>0:V_t<a\}, \quad a >0.
\]


Our derivation of the results relies on the following remark on the connection with the drift-changed process $Y$ and the reflected process $U$.
\begin{remark} \label{remark_connection_Y_U} Recall the hitting times $\tau_b^-$ of $Y$  and $\kappa_b^+$ of $U$ as in \eqref{first_passage_time} and \eqref{def_kappa_time}, respectively.  Almost surely, under $\p_x$ for any $x \in \R$, we have $T_b^+ =\kappa_b^+$ and $V_t = U_t$ on $0 \leq t \leq T_b^+$; similarly, we have $T_b^- =\tau_b^-$ and $V_t = Y_t$  on $0 \leq t < T_b^-$ and $V_{T_b^--} + \Delta X_{T_b^-} = Y_{T_b^-}$ on $\{ T_b^- < \infty\}$.
\end{remark}



\begin{theorem}[Resolvent]\label{resol}
For $q \geq 0$, $x \leq a$ and a Borel set $B$ on $[0,a]$,
\begin{align}\label{resol1}
\begin{split}
\e_x\bigg(\int_0^{T_a^+}e^{-qt}1_{\{ V_t \in B \}}\ud t\bigg) = \int_{B} \Big( w^{(q)}(a,z)\frac{r^{(q)}(x)}{r^{(q)}(a)} - w^{(q)}(x,z) \Big) \diff z,
\end{split}
\end{align}
where, for all $0 \leq z \leq a$, 
\begin{align*}
w^{(q)}(x, z) 
&:= 1_{\{ 0 < z < b \}} \Big(W^{(q)}(x-z) +\delta \int_{b}^x\mathbb{W}^{(q)}(x-y)W^{(q)\prime}(y-z)\ud y \Big)+ 1_{\{ b < z < x\}} \mathbb{W}^{(q)}(x-z).
\end{align*}
\end{theorem}

\begin{proof}
We shall prove the result for $q > 0$ because the case $q = 0$ can be obtained by monotone convergence and the continuity of the scale function in $q$ as in Lemma 8.3 of \cite{K}.

(i) Suppose $X$ is of bounded variation, and
let $f^{(q)}(x,a;B)$ be the left hand side of \eqref{resol1}.  

For $x< b$, using the strong Markov property, Remark \ref{remark_connection_Y_U}, \eqref{rsr}, and \eqref{upcrossing_time_reflected}, 
\begin{align}\label{iden6}
f^{(q)}(x,a;B)&=\e_x\bigg(\int_0^{\kappa_b^+}e^{-qt}1_{\{U_t\in B\}}\ud t\bigg)+\e_x\big(e^{-q\kappa_b^+}\big) f^{(q)}(b,a;B)\notag\\
&=\frac{Z^{(q)}(x)}{Z^{(q)}(b)} [\rho^{(q)}(b;B)+ f^{(q)}(b,a;B)] -\rho^{(q)}(x;B).\end{align}
On the other hand, for $x \geq b$, again by the strong Markov property and Remark \ref{remark_connection_Y_U}, together with \eqref{resolvent_density} and \eqref{undershoot_expectation}, 
\begin{align}\label{potentiala}
\begin{split}
f^{(q)}&(x,a;B)=\E_x\Big(\int_0^{\tau_a^+\wedge\tau_b^-}e^{-qt}1_{\{Y_t\in B\}}\ud t\Big)+\E_x\left(e^{-q\tau_b^-}f^{(q)}(Y_{\tau_b^-},a;B)1_{\{\tau_b^-<\tau_a^+\}}\right)\\
&=\int_{b}^{a}1_{\{y\in B\}}\left\{\frac{\mathbb{W}^{(q)}(x-b)\mathbb{W}^{(q)}(a-y)}{\mathbb{W}^{(q)}(a-b)}-\mathbb{W}^{(q)}(x-y)\right\}\ud y \\
&+\int_0^{a-b}\int_{(-\infty,-y)} f^{(q)}(y+u+b,a;B)\left\{\frac{\mathbb{W}^{(q)}(x-b)\mathbb{W}^{(q)}(a-b-y)}{\mathbb{W}^{(q)}(a-b)}-\mathbb{W}^{(q)}(x-b-y)\right\}\Pi(\ud u) \ud y.
\end{split}
\end{align}
Here, by setting $x=b$, dividing both sides by $\mathbb{W}^{(q)}(0)/\mathbb{W}^{(q)}(a-b)= [(c-\delta) \mathbb{W}^{(q)}(a-b)]^{-1}$, and substituting (\ref{iden6}), 
\begin{align} \label{recursion_W_time_v}
\begin{split}
&(c-\delta) {\mathbb{W}^{(q)}(a-b)} f^{(q)}(b,a;B)\\ 
&=\int_{b}^{a}1_{\{y\in B\}} \mathbb{W}^{(q)}(a-y)  \ud y \\
&+\int_0^{a-b}\int_{(-\infty,-y)}\Big(\frac{Z^{(q)}(y+u+b)}{Z^{(q)}(b)} [\rho^{(q)}(b;B) + f^{(q)}(b,a; B)]-\rho^{(q)}(y+u+b;B) \Big) {\mathbb{W}^{(q)}(a-b-y)}\Pi(\ud u) \ud y.
\end{split}
\end{align}
In particular, by Fubini's theorem and \eqref{useful_identity_W},
\begin{align} \label{varphi_time_W_formula}
\begin{split}
&\int_0^{a-b} \int_{(-\infty,-y)}\rho^{(q)}(y+u+b;B)\mathbb{W}^{(q)}(a-b-y)  \Pi (\ud u) \ud y\\
&=\int_0^{\infty}\int_{(-\infty,-y)}\int_0^{\infty}W^{(q)}(y+u+b-z)\mathbb{W}^{(q)}(a-b-y)1_{\{z\in B\}}\ud z\Pi(\ud u)\ud y\\
&=\int_0^b 1_{\{z\in B\}}\int_0^{\infty}\int_{(-\infty,-y)}W^{(q)}(y+u+b-z)\mathbb{W}^{(q)}(a-b-y)\Pi(\ud u)\ud y \ud z \\
&=\int_0^b 1_{\{z\in B\}}\left((c-\delta) W^{(q)}(b-z)\mathbb{W}^{(q)}(a-b)-W^{(q)}(a-z)-\delta\int_{b}^{a}\mathbb{W}^{(q)}(a-y)W^{(q)\prime}(y-z)\ud y\right) \ud z.
\end{split}
\end{align}
By substituting \eqref{lemma_useful_identity_c} and \eqref{varphi_time_W_formula} in \eqref{recursion_W_time_v} and solving for $f^{(q)}(b,a;B)$, we obtain
\begin{align}\label{potentialb}
f^{(q)}(b,a;B)&=Z^{(q)}(b)\frac{\int_{B} w^{(q)}(a,z) \ud z}{r^{(q)}(a)}-\rho^{(q)}(b;B).
\end{align}

Substituting (\ref{potentialb}) in (\ref{iden6}), the claim holds for $x < b$. 
\par For $x > b$, the equation (\ref{potentiala}) gives
\begin{align} \label{f_above_b}
\begin{split}
f^{(q)}(x,a;B)
&=\int_b^a 1_{\{y\in B\}}\left\{\frac{\mathbb{W}^{(q)}(x-b)\mathbb{W}^{(q)}(a-y)}{\mathbb{W}^{(q)}(a-b)}-\mathbb{W}^{(q)}(x-y)\right\}\ud y \\
&+\int_0^{a-b}\int_{(-\infty,-y)}\int_B \Big[w^{(q)}(a,z)\frac{Z^{(q)}(y+u+b)}{r^{(q)}(a)} - W^{(q)}(y+u+b-z)\Big] \diff z \\ &\qquad \times \left\{\frac{\mathbb{W}^{(q)}(x-b)\mathbb{W}^{(q)}(a-b-y)}{\mathbb{W}^{(q)}(a-b)}-\mathbb{W}^{(q)}(x-b-y)\right\}\Pi(\ud u) \ud y 
\end{split}
\end{align}
where  the second integral equals, by  \eqref{useful_identity_W}, \eqref{cc0}, and \eqref{varphi_time_W_formula},
\begin{align*}
&\frac{\int_B w^{(q)}(a,z)\diff z}{r^{(q)}(a)}\int_0^{a-b}\int_{(-\infty,-y)}Z^{(q)}(y+u+b)\left\{\frac{\mathbb{W}^{(q)}(x-b)\mathbb{W}^{(q)}(a-b-y)}{\mathbb{W}^{(q)}(a-b)}-\mathbb{W}^{(q)}(x-b-y)\right\}\Pi(\ud u) \ud y\notag\\
&- \int_0^{a-b}\int_{(-\infty,-y)}\rho^{(q)}(y+u+b;B)\left\{\frac{\mathbb{W}^{(q)}(x-b)\mathbb{W}^{(q)}(a-b-y)}{\mathbb{W}^{(q)}(a-b)}-\mathbb{W}^{(q)}(x-b-y)\right\}\Pi(\ud u) \ud y\notag\\
&= \frac{\int_B w^{(q)}(a,z)\diff z}{r^{(q)}(a)}\left(r^{(q)}(x)-\frac{\mathbb{W}^{(q)}(x-b)}{\mathbb{W}^{(q)}(a-b)}r^{(q)}(a)\right)\notag\\
&+ \frac{\mathbb{W}^{(q)}(x-b)}{\mathbb{W}^{(q)}(a-b)} \int_0^b 1_{\{ z \in B\}}\left(W^{(q)}(a-z)+\delta\int_{b}^{a}\mathbb{W}^{(q)}(a-y)W^{(q)\prime}(y-z)\diff y\right) \diff z \notag\\
&- \int_0^b 1_{\{z \in B\}}\left(W^{(q)}(x-z)+\delta\int_{b}^{x}\mathbb{W}^{(q)}(x-y)W^{(q)\prime}(y-z)\diff y\right)\diff z.
\end{align*}
Substituting these identities in \eqref{f_above_b} and after some computations, we obtain that  the term involving ${\mathbb{W}^{(q)}(x-b)} / {\mathbb{W}^{(q)}(a-b)}$ vanishes, and the claim follows.

(ii) We now extend the result to the case of unbounded variation. 
 Let $V^{(n)}$ be the refracted-reflected process associated to $X^{(n)}$ and $T_{a,n}^+$ its first passage time above $a$,  of a strongly approximating sequence (of bounded variation) for $X$ as in Proposition \ref{prop_approximation}.

First, it can easily be shown that $T_{a,n}^+$ and $T_a^+$ are both finite a.s.\ because these are bounded from above by the upcrossing times at $a$ for the  reflected process of the drift-changed process $Y$ (without refraction) as in \eqref{reflected_levy} with $X$ replaced with $Y$ (which are known to be finite a.s.); this will be confirmed in  Corollary \ref{corollary_one_sided} below.

Now in order to show $T_{a,n}^+ \xrightarrow{n \uparrow \infty} T_a^+$  holds a.s, as in page 212 of \cite{B}, it suffices to show
\begin{align}
T_{a-}^+ = T_{a}^+ \quad a.s. \label{T_a_cont}
\end{align}
Indeed, for $0 < \varepsilon < a-b$, because when the process takes values in $(b,a)$ then $V$ is a spectrally negative \lev process $Y$, we have
\begin{align}
0 \leq T_{a}^+ - T_{a-\varepsilon}^+ \leq (\tau_a^+ 1_{\{ \tau_b^- > \tau_a^+\}}) \circ \theta_{T^+_{a-\varepsilon}} +  (1_{\{ \tau_b^- < \tau_a^+ \}}   \circ \theta_{T^+_{a-\varepsilon}}) T_a^+,\label{diff_T_a}
\end{align}
where $\theta$ is the time-shift operator. The process  $\{ \tau_s^+; s \geq a- \varepsilon \}$ is a subordinator (that is a.s.\ continuous at time $a$) and hence 
$(\tau_a^+ 1_{\{ \tau_b^- > \tau_a^+\}}) \circ \theta_{T^+_{a-\varepsilon}}$ vanishes in the limit as $\varepsilon \downarrow 0$. On the other hand, 
 by the regularity of the spectrally negative \lev process $Y$, we have $1_{\{ \tau_b^- < \tau_a^+ \}}  \circ \theta_{T^+_{a-\varepsilon}} \rightarrow 0$ as $\varepsilon \downarrow 0$, showing (together with the finiteness of $T_a^+$) that the second term on the right hand side of \eqref{diff_T_a} vanishes as well. In sum,  we have \eqref{T_a_cont} and the convergence $T_{a,n}^+ \xrightarrow{n \uparrow \infty} T_a^+$ holds.

Now, by Remark  \ref{remark_strongly_approximating} and dominated convergence, it is sufficient to show that $\p_x ( V_t \in \partial B)$ and $\p_x( \sup_{0\leq s \leq t}V_s = a )$ vanish for Lebesgue a.e.\ $t > 0$; similar results have been  obtained in the proof of Theorem 6 of  \cite{KL} for the refracted \lev process.
For the former, we give the following result.
 \begin{lemma}  \label{lemma_mass_zero}We have $\p_x ( V_t = y)=0$  for $y \in [0, \infty)$ and Lebesgue a.e.\ $t > 0$.
\end{lemma}
\begin{proof}  We will show that the result holds for the case $x < b$; the case $x \geq b$ can be shown similarly.
\par Let us define a sequence of stopping times $0 =: S_b^-(0) < S_b^+(1) < S_b^- (1) < S_b^+(2) < S_b^-(2) < \cdots < S_b^+(n) < S_b^- (n) < \cdots$ as follows: 
\begin{align*}
S_b^+(n)&:=\inf\{t > S_b^-(n-1):V_t > b\} \text{ and } S_b^-(n):=\inf\{t > S_b^+(n):V_t =0\}, \quad n \geq 1.
\end{align*}
Then we have
\begin{align*}
\mathbb{P}_x(V_t=y)&= \p_x \left(V_t=y, t\in[0,S_b^+(1))\right) \\
&+\sum_{n\geq 1}\Big[ \mathbb{P}_x\left(V_t=y, t\in[S_b^+(n),S_b^-(n))\right) + \mathbb{P}_x\left(V_t=y, t\in[S_b^-(n),S_b^+(n+1))\right) \Big].\end{align*}
By Remark \ref{remark_connection_Y_U} and \eqref{rsr}, we have $\p_x \left(V_t=y, t\in[0,S_b^+(1))\right) = \p_x \left(U_t=y, t \in[0,\kappa_b^+)\right) = 0$ for a.e. $t > 0$.
In addition,
\begin{align*}
\mathbb{P}_x\left(V_t=y, t\in[S_b^+(n),S_b^-(n))\right) = \mathbb{P}_x\left[ \p_x (V_t=y, t\in[S_b^+(n),S_b^-(n)) | S_b^+(n)  ) \right]. \end{align*}
With $\tilde{A}_{t}$ being a refracted process starting at $b$ that is independent of $\mathcal{F}_{S_b^+(n)}$,
\begin{align*}
\p_x (V_t=y, t\in[S_b^+(n),S_b^-(n)) |S_b^+(n)  )  &\leq \p_x (\tilde{A}_{t -S_b^+(n)}=y, t\in[S_b^+(n),\infty) | S_b^+(n)  ) \\&= \p_b ( A_{t-s} = y ) |_{s =S_b^+(n)}.
\end{align*}
Hence,
\begin{align*}
\mathbb{P}_x\left(V_t=y, t\in[S_b^+(n),S_b^-(n))\right) \leq \int_0^t \p_x ( S_{b}^+(n) \in \diff s ) \p_b (A_{t-s} = y) = 0.\end{align*}
Similarly, we have $\mathbb{P}_x\left(V_t=y, t\in[S_b^-(n),S_b^+(n+1))\right) = 0$. Summing up these, we have that $\p_x (V_t = y) = 0$ for a.e. $t > 0$.
\end{proof}

By this lemma, $\p_x \{ V_t \in \partial B\} = 0$ for a.e.\ $t > 0$. The proof of $\p_x\{ \sup_{0 \leq s \leq t}V_s= a \} = 0$ can be similarly done using the fact that $\p \{ \sup_{0 \leq s \leq t}A_s = a\}=0$ for a.e. $t > 0$ as proved in Theorem 6 of \cite{KL}.

\end{proof}

\begin{remark}
In particular, when $\delta = 0$ or $a=b$ in Theorem \ref{resol},  we recover \eqref{rsr}.
\end{remark}
By taking $a$ to $\infty$ in the previous result,  we shall obtain the following.
\begin{corollary}\label{corollaryresol1}
Fix $x \in \R$ and any Borel set $B$ on $[0,\infty)$.

(i) For $q > 0$,
\begin{align*} 
\e_x\bigg(\int_0^{\infty}e^{-qt}1_{\{ V_t \in B \}}\ud t\bigg) &= \int_B \Big( \frac { e^{-\varphi(q)z}1_{\{ b < z \}} +\delta 1_{\{ 0 < z < b \}}\int_{b}^{\infty}e^{-\varphi(q)u}W^{(q)\prime}(u-z)\ud u  }  {\delta q \int_b^\infty e^{-\varphi(q)y}W^{(q)}(y)\ud y}{r^{(q)}(x)} - w^{(q)}(x,z) \Big) \diff z. 
\end{align*}
(ii) For $q = 0$ and $\psi_Y'(0+) > 0$ (or $Y_t \xrightarrow{t \uparrow \infty} \infty$ a.s.), then
\begin{align*}
\e_x\bigg(\int_0^{\infty} 1_{\{ V_t \in B \}}\ud t\bigg) &= \int_B \Big( \frac { 1_{\{ b < z \}} + 1_{\{ 0 < z < b \}} \big( 1- \delta^{-1}W (b-z)   \big)  }  {\psi_Y'(0+)} - w^{(0)}(x,z) \Big)\diff z. 
\end{align*}
For $q = 0$ and  $\psi_Y'(0+) \leq 0$, it becomes infinity given $Leb(B) > 0$.
\end{corollary}
\begin{proof}
(i) We have
\begin{align*}
\frac {w^{(q)}(a, z)} {r^{(q)}(a)} &=  \frac {\mathbb{W}^{(q)}(a)^{-1} w^{(q)}(a, z)} {\mathbb{W}^{(q)}(a)^{-1} r^{(q)}(a)} \xrightarrow{a \uparrow \infty}  \frac { e^{-\varphi(q)z}1_{\{ b < z \}} +\delta 1_{\{ 0 < z < b \}}\int_{b}^{\infty}e^{-\varphi(q)u}W^{(q)\prime}(u-z)\ud u  }  {\delta q \int_b^\infty e^{-\varphi(q)y}W^{(q)}(y)\ud y}.\end{align*}
Here the convergence follows by dominated convergence thanks to the bound $\mathbb{W}^{(q)}(a-y)/\mathbb{W}^{(q)}(a-b) \leq  e^{-\varphi(q)(y-b)}$ by \eqref{W_q_limit}, and from Exercise 8.5 (i) in \cite{K} and by \eqref{W_q_limit}, that for all $z \in \R$, 
\[
\lim_{a\to\infty}\frac{Z^{(q)}(a)}{W^{(q)}(a)}=\frac{q}{\Phi(q)}, \quad\lim_{a\to\infty}\frac{\mathbb{W}^{(q)}(a-z)}{\mathbb{W}^{(q)}(a)}=e^{-\varphi(q)z}, \quad
\lim_{a\to\infty}\frac{W^{(q)}(a-z)}{\mathbb{W}^{(q)}(a)}=\lim_{a\to\infty}\frac{Z^{(q)}(a-z)}{\mathbb{W}^{(q)}(a)}=0.
\]

(ii) By monotone convergence we shall take $q \rightarrow 0$ in (i). We have
\begin{align*}
\delta q \int_b^\infty e^{-\varphi(q)y}W^{(q)}(y)\ud y = \delta q \Bigg(  \frac 1 {\psi(\varphi(q))-q }- \int_0^b e^{-\varphi(q)y}W^{(q)}(y)\ud y \Bigg) = \frac q {\varphi(q)} - \delta q \int_0^b e^{-\varphi(q)y}W^{(q)}(y)\ud y. 
\end{align*}
When $\psi_Y'(0+) \geq 0$, then $\varphi(0) = 0$ and hence $\delta q \int_b^\infty e^{-\varphi(q)y}W^{(q)}(y)\ud y \xrightarrow{q \downarrow 0} \psi_Y'(0+)$;
when $\psi_Y'(0+) < 0$, then $\varphi(0) > 0$ and $\delta q \int_b^\infty e^{-\varphi(q)y}W^{(q)}(y)\ud y \xrightarrow{q \downarrow 0} 0$.
On the other hand, using (8.20) of \cite{K},
\begin{align*}
\int_{b}^{\infty}e^{-\varphi(q)u}W^{(q)\prime}(u-z)\ud u &= e^{-\varphi(q) z}\int_{b-z}^{\infty}e^{-\varphi(q)y}W^{(q)\prime}(y)\ud y \\ &= e^{-\varphi(q) z} \Big( \frac {\varphi(q)} {\psi(\varphi(q))-q} - W^{(q)} (0)  - \int^{b-z}_0e^{-\varphi(q)y}W^{(q)\prime}(y)\ud y \Big) 
\\ &= e^{-\varphi(q) z} \Big( \delta^{-1} - W^{(q)} (0)  - \int^{b-z}_0e^{-\varphi(q)y}W^{(q)\prime}(y)\ud y \Big) \\
&\xrightarrow{q \downarrow 0} e^{-\varphi(0) z} \Big( \delta^{-1} - W (0)  - \int^{b-z}_0e^{-\varphi(0)y}W^{\prime}(y)\ud y \Big),
\end{align*}
which simplifies to $\delta^{-1} - W (b-z)$ when $\varphi(0) = 0$.
Combining these, we have the result for $q = 0$.
\end{proof}

Another corollary can be obtained by setting $B = [0,a]$ in \eqref{resol1}.
\begin{corollary}[One-sided exit] \label{corollary_one_sided}
For any $q\geq0$ and $x \leq a$, we have 
\begin{align}
\mathbb{E}_x\left(e^{-qT_a^+} \right)=\frac{r^{(q)}(x)}{r^{(q)}(a)}.\notag
\end{align}
In particular, $T_a^+ < \infty$ $\p_x$-a.s.
\end{corollary}
\begin{proof}
(i) By \eqref{RLqp}, we get
\begin{align*}
\delta\int_b^a\int_0^b\mathbb{W}^{(q)}(a-u)W^{(q)\prime}(u-z)\ud z \ud u &= -\delta\int_b^a\mathbb{W}^{(q)}(a-u)W^{(q)}(u-b)\ud u+\delta\int_b^a\mathbb{W}^{(q)}(a-u)W^{(q)}(u)\ud u \\
&= -\overline{\mathbb{W}}^{(q)}(a-b) + \overline{W}^{(q)}(a-b) +\delta\int_b^a\mathbb{W}^{(q)}(a-u)W^{(q)}(u)\ud u. 
\end{align*}
Hence,
\begin{align*}
\int_0^a w^{(q)}(a, z) \ud z 
&=\overline{W}^{(q)}(a) + \delta \int_b^a\mathbb{W}^{(q)}(a-u)W^{(q)}(u)\ud u = \frac {r^{(q)}(a)-1} q.
\end{align*}
Therefore, substituting the above expression in \eqref{resol1} with $B = [0,a]$, we get
\begin{multline*}
\e_x\bigg(\int_0^{T_a^+}e^{-qt}\ud t\bigg) = \bigg(\int_0^a w^{(q)}(a,z) \diff z \bigg) \frac{r^{(q)}(x)}{r^{(q)}(a)} - \int_0^x w^{(q)}(x,z)  \diff z \\
=\left( \frac {r^{(q)}(a)-1} q\right) \frac{r^{(q)}(x)}{r^{(q)}(a)} - \frac {r^{(q)}(x)-1} q  = q^{-1} \Big( 1- \frac {r^{(q)}(x)} {r^{(q)}(a)}\Big).
\end{multline*}
The result follows by noting that $\mathbb{E}_x\big(e^{-qT_a^+}\big)=1-q\mathbb{E}_x (\int_0^{T_a^+} e^{-qt} \ud t )$.  

(ii) The finiteness of $T_a^+$ holds by setting $q = 0$ and noting that $r^{(0)}(x)=r^{(0)}(a)=1$.


\end{proof}
\begin{remark}
In particular, when $\delta = 0$ in Corollary \ref{corollary_one_sided}, we recover \eqref{upcrossing_time_reflected}. 
\end{remark}


We conclude this section with the expression of the expected NPV of $L$ as discussed in Section \ref{subsection_dividends}; these are immediate by Theorem \ref{resol} and Corollary \ref{corollaryresol1}.
\begin{corollary}
For any $q\geq0$ and $x\leq a$, we have
\begin{align*}
\mathbb{E}_x\left(\int_0^{T_a^+}e^{-qt}\ud L_t\right) 
&= \delta \overline{\mathbb{W}}^{(q)}(a-b) \frac{r^{(q)}(x)}{r^{(q)}(a)}- \delta \overline{\mathbb{W}}^{(q)}(x-b). 
\end{align*}
\end{corollary}
\begin{corollary} \label{cor_dividend_infty}
 Fix $x \in \R$.  
For $q > 0$, we have
\begin{align}
\mathbb{E}_x&\left(\int_0^{\infty}e^{-qt}\ud L_t\right)={e^{-\varphi(q)b}} \frac { r^{(q)}(x)} {\displaystyle \varphi(q)q\int_b^{\infty}e^{-\varphi(q)y}W^{(q)}(y)\ud y} - \delta \overline{\mathbb{W}}^{(q)}(x-b).\notag
\end{align}
For $q = 0$, it becomes infinity.
\end{corollary}
\section{Costs of capital injection} \label{section_capital_injection}
In this section, we derive the expression for the second item of \eqref{net_present_value}, which corresponds to the NPV of capital injection in the insurance context.

\begin{proposition} \label{prop_capital_injection}Suppose $\psi'(0+) > -\infty$ and $q > 0$.
For any $x\leq a$, we have 
\begin{align}\label{capital cost}
\begin{split}
\mathbb{E}_x&\left(\int_{[0,T_a^+]}e^{-qt}\ud R_t\right)=\tilde{r}^{(q)}(a) \frac{r^{(q)}(x)} {r^{(q)}(a)}- \tilde{r}^{(q)}(x),
\end{split}
\end{align}
where
\begin{align*}
\tilde{r}^{(q)}(x)
&:= \overline{Z}^{(q)}(x) + \frac {\psi'(0+)} q +\delta\int_b^x\mathbb{W}^{(q)}(x-y)Z^{(q)}(y)\ud y, \quad x \in \R.
\end{align*} 
\end{proposition}

\begin{proof} (i) Assume that $X$ is of bounded variation, and let $g^{(q)}(x,a)$ be the left hand side of \eqref{capital cost}. For $x<b$, by an application of the strong Markov property, Remark \ref{remark_connection_Y_U}, \eqref{upcrossing_time_reflected}, and \eqref{capital_injection_identity_SN}, \begin{align}\label{12}
g^{(q)}(x,a)&=\mathbb{E}_x\Big(\int_{[0,T_b^+]} e^{-qt}\ud R_t\Big)+\mathbb{E}_x(e^{-qT_b^+})g^{(q)}(b,a)\notag\\
&=-\Big( \overline{Z}^{(q)}(x)+\frac{\psi'(0+)}{q} \Big)+\frac{Z^{(q)}(x)}{Z^{(q)}(b)}\left(\overline{Z}^{(q)}(b)+\frac{\psi'(0+)}{q}+g^{(q)}(b,a)\right).
\end{align} 
\par Now in the case where $x\geq b$, we obtain using (\ref{12}),  Remark \ref{remark_connection_Y_U}, and the fact that $R$ stays constant on $[0, T_b^-)$,

\begin{align} \label{g_x_a_probabilistic}
\begin{split}
g^{(q)}(x,a)
&=\mathbb{E}_x \big(e^{-q\tau_b^-}g^{(q)}(Y_{\tau_b^-},a) 1_{\{\tau_b^-<\tau_a^+\}} \big)\\
&=\mathbb{E}_x\left(e^{-q\tau_b^-}\left\{-\overline{Z}^{(q)}(Y_{\tau_b^-})-\frac{\psi'(0+)}{q}+\frac{Z^{(q)}(Y_{\tau_b^-})}{Z^{(q)}(b)}\left(\overline{Z}^{(q)}(b)+\frac{\psi'(0+)}{q}+g^{(q)}(b,a)\right)\right\} 1_{\{\tau_b^-<\tau_a^+\}}\right).
\end{split}
\end{align}
Here, by using \eqref{undershoot_expectation} and \eqref{lemma_useful_identity_Z_bar}, we obtain 
\begin{align}\label{cc1}
\begin{split}
\mathbb{E}_x&\left(e^{-q\tau_b^-}\overline{Z}^{(q)}(Y_{\tau_b^-}) 1_{\{\tau_b^-<\tau_a^+ \}}\right)
\\
&=\int_0^{a-b}\int_{(-\infty,-y)}\overline{Z}^{(q)}(y+u+b)\left\{\frac{\mathbb{W}^{(q)}(x-b)\mathbb{W}^{(q)}(a-b-y)}{\mathbb{W}^{(q)}(a-b)}-\mathbb{W}^{(q)}(x-b-y)\right\}\Pi(\ud u)\ud y \\
&=\overline{Z}^{(q)}(x)+\delta\int_b^x\mathbb{W}^{(q)}(x-y)Z^{(q)}(y)\ud y-\psi'(0+)\overline{\mathbb{W}}^{(q)}(x-b)\\
&-\frac{\mathbb{W}^{(q)}(x-b)}{\mathbb{W}^{(q)}(a-b)}\Big(\overline{Z}^{(q)}(a)+\delta\int_b^a\mathbb{W}^{(q)}(a-y)Z^{(q)}(y)\ud y- \psi'(0+)\overline{\mathbb{W}}^{(q)}(a-b)\Big).
\end{split}
\end{align}
Substituting  \eqref{laplace_in_terms_of_z}, \eqref{cc0}, and \eqref{cc1} in \eqref{g_x_a_probabilistic}, we get, for $x\geq b$,	
\begin{align} \label{g_x_a_above_b}
g^{(q)}(x,a)
&=-\tilde{r}^{(q)}(x)+\frac{\mathbb{W}^{(q)}(x-b)}{\mathbb{W}^{(q)}(a-b)}\tilde{r}^{(q)}(a) 
+\frac {\tilde{r}^{(q)}(b)+g^{(q)}(b,a)} {Z^{(q)}(b)} \left(r^{(q)}(x)-\frac{\mathbb{W}^{(q)}(x-b)}{\mathbb{W}^{(q)}(a-b)} r^{(q)}(a)\right).
\end{align}

Setting $x=b$, we obtain
\begin{align*}
g^{(q)}(b,a)
&=-\tilde{r}^{(q)}(b)+\frac {\tilde{r}^{(q)}(a)} { (c-\delta) \mathbb{W}^{(q)}(a-b)} +\frac {\tilde{r}^{(q)}(b)+g^{(q)}(b,a) } {Z^{(q)}(b)} \left(Z^{(q)}(b)-\frac {r^{(q)}(a)} {(c-\delta)\mathbb{W}^{(q)}(a-b)} \right).
\end{align*}
Hence we have 
\begin{align}\label{above expression}
g^{(q)}(b,a)=\frac{\tilde{r}^{(q)}(a)}{r^{(q)}(a)}Z^{(q)}(b)-\tilde{r}^{(q)}(b).
\end{align}
By substituting (\ref{above expression}) in \eqref{12}, we obtain \eqref{capital cost} for $x< b$; on the other hand, by using (\ref{above expression}) in \eqref{g_x_a_above_b}, we get \eqref{capital cost} for $x \geq b$.

(ii) We now extend this result to the case $X$ is of unbounded variation. 
Let $V^{(n)}$, $Y^{(n)}$, $R^{(n)}$, $L^{(n)}$, and $T_{a,n}^+$ are those for $X^{(n)}$ (of bounded variation) of a strongly approximating sequence for $X$. Without loss of generality, we can choose such sequence so that the Poisson random measure $N^{(n)}(\diff s, \diff u)$ of $X^{(n)}$, for all $n \in \mathbb{N}$, coincides with the Poisson random measure $N(\diff s, \diff u)$ of $X$ for $u < -1$ a.s.   Recall also, as in the proof of Theorem \ref{resol}, that $T_{a,n}^+ \xrightarrow{n \uparrow \infty} T_a^+$ a.s. 

Define, for $n \in \mathbb{N}$,
\begin{align*} 
g_n^{(q)}(x,a) := \mathbb{E}_x\Big(\int_{[0,T_{a,n}^{+}]}e^{-qt}\ud R_t^{(n)}\Big).
\end{align*}
In view of the expression \eqref{capital cost} for the bounded variation case and by Remark \ref{remark_strongly_approximating}, we have
\begin{align*}
\|g \| := \sup_{n\in\mathbb{N}} \sup_{0 \leq y \leq a}g^{(q)}_n(y,a)  < \infty.
\end{align*}

We define $\tau_{-M,n}^{-} := \inf \{ t > 0: Y_t^{(n)} < - M\}$ and consider the decomposition, for $M > 0$, 
\begin{align} \label{g_decomposition}
g^{(q)}_n(x,a) &=\E_x \Big( \int_{[0, T_{a,n}^+ \wedge \tau_{-M,n}^-]} e^{-qs} \diff R_s^{(n)} \Big)+\E_x \Big( 1_{\{T_{a,n}^+ > \tau_{-M,n}^-\}}  \int_{(\tau_{-M,n}^-, T_{a,n}^+]} e^{-qs} \diff R_s^{(n)}\Big).
\end{align}
(1) We shall first show that the second expectation can be made arbitrarily small uniformly in $n \geq N$ by choosing $M, N$ sufficiently large.
We have, with $(\mathcal{F}_t^{(n)}; t \geq 0)$ being the natural filtration of $X^{(n)}$,
\begin{align} \label{bound_R_tail}
\begin{split}
\E_x \Big( 1_{\{T_{a,n}^+ > \tau_{-M,n}^-\}} \int_{(\tau_{-M,n}^-, T_{a,n}^+]} e^{-qs} \diff R_s^{(n)} \Big) &= \E_x \Big[ \E_x \Big( 1_{\{T_{a,n}^+ > \tau_{-M,n}^-\}} \int_{(\tau_{-M,n}^-, T_{a,n}^+]} e^{-qs} \diff R_s^{(n)} \Big| \mathcal{F}_{\tau_{-M,n}^-}^{(n)}\Big)\Big] \\
&\leq \E_x\big( e^{-q \tau_{-M,n}^-}1_{\{ \tau_{-M,n}^-<T_{a,n}^+\}} \big) \|g \|.
 \end{split}
\end{align}
Bounded convergence gives 
\begin{align*}
\E_x \big( e^{-q \tau_{-M,n}^-}1_{\{ \tau_{-M,n}^-<T_{a,n}^+\}} \big) &\xrightarrow{n \uparrow \infty}\E_x \big( e^{-q \tau_{-M}^-}1_{\{ \tau_{-M}^-<T_a^+\}} \big), \\
\E_x \big( e^{-q \tau_{-M}^-}1_{\{ \tau_{-M}^-<T_a^+\}} \big) \leq \E_x \big( e^{-q \tau_{-M}^-}1_{\{ \tau_{-M}^-< \infty\}} \big) &\xrightarrow{M \uparrow \infty} 0.
\end{align*}
This means, for any arbitrary $\varepsilon > 0$, we can choose sufficiently large $N \in \mathbb{N}$ and $M > 0$ such that
\[
\sup_{n \geq N} \E_x \big( e^{-q \tau_{-M,n}^-}1_{\{ \tau_{-M,n}^-<T_{a,n}^+\}} \big)  < \varepsilon / \|g \|. \]
Namely, the left hand side of \eqref{bound_R_tail} can be made arbitrarily small by choosing sufficiently large $N$ and $M$.

(2) Consider now the first expectation in \eqref{g_decomposition}. On $[0, T_{a,n}^+]$, we have $V_t^{(n)} = X_t^{(n)} + R_t^{(n)} - L_t^{(n)} \leq a$ and hence 
\begin{align}
R_t^{(n)} \leq a + L_t^{(n)} - X_t^{(n)} \leq a - Y_t^{(n)} \leq a - \inf_{0 \leq s \leq t}Y_s^{(n)}. \label{R_bound}
\end{align}
On $[0,  T_{a,n}^+ \wedge \tau_{-M,n}^{-})$, $R_t^{(n)} \leq a + M$ and therefore $\int_{[0, T_{a,n}^+ \wedge \tau_{-M,n}^{-})} e^{-qs} \diff R_s^{(n)}$ is bounded.
Furthermore, by noting that $\Delta R_{T^+_{a,n}}^{(n)} = 0$ and by \eqref{R_bound},
\begin{multline*}
\int_{\{ T_{a,n}^+ \wedge \tau_{-M,n}^{-} \}} e^{-qs} \diff R_s^{(n)}  =\int_{\{\tau_{-M,n}^{-}\}} e^{-qs} \diff R_s^{(n)} 1_{\{\tau_{-M,n}^{-} < T_{a,n}^+\}} \\ \leq e^{-q \tau_{-M,n}^-} \big|R_{\tau_{-M,n}^-}^{(n)} \big| 1_{\{\tau_{-M,n}^{-} < T_{a,n}^+\}}
 \leq e^{-q \tau_{-M,n}^-} (\big|Y_{\tau_{-M,n}^-}^{(n)} \big| + a).
\end{multline*}
By $N(\diff s, \diff u) = N^{(n)}(\diff s, \diff u)$ for $u < -1$ and with $\underline{Y}$ the running infimum process of $Y$,  we have uniformly in $n \in \mathbb{N}$
\begin{align*}
e^{-q\tau_{-M,n}^-} \big|Y_{\tau_{-M,n}^-}^{(n)} \big|1_{\big\{Y_{\tau_{-M,n}^-}^{(n)} < -M-1\big\}} 
&= \int_0^\infty \int_{(-\infty,0)} e^{-qs} (-u-Y_{s-}^{(n)} ) 1_{ \{\underline{Y}_{s-}^{(n)} > -M, Y_{s-}^{(n)} + u < - M-1 \} } N^{(n)}(\diff s,  \diff u) \\
&\leq \int_0^\infty \int_{(-\infty, -1)} e^{-qs} (|u|+M) 1_{ \{\underline{Y}_{s-}^{(n)} > -M, Y_{s-}^{(n)} + u < - M-1 \} } N^{(n)}(\diff s, \diff u) \\ &= \int_0^\infty \int_{(-\infty,-1)} e^{-qs} (|u|+M) 1_{ \{\underline{Y}_{s-}^{(n)} > -M, Y_{s-}^{(n)} + u < - M-1 \} } N(\diff s, \diff u) \\
&\leq \int_0^\infty \int_{(-\infty,-1)} e^{-qs} (|u|+M)  N(\diff s,  \diff u), 
\end{align*}
which is integrable because, by the assumption that $\psi'(0+) > -\infty$,
\begin{align*}
\E_x \Big( \int_0^\infty \int_{(-\infty,-1)} e^{-qs} |u|  N(\diff s, \diff u) \Big) = \E_x \Big( \int_0^\infty \int_{(-\infty,-1)} e^{-qs} |u|  \Pi(\diff u)  \diff s \Big) = q^{-1} \int_{(-\infty,-1)}  |u|  \Pi(\diff u) < \infty.
\end{align*}
In sum, $\int_{[0, T_{a,n}^+ \wedge \tau_{-M,n}^{-}]} e^{-qs} \diff R_s^{(n)}$ is bounded in $n$ by an integrable random variable. Hence, Fatou's lemma gives
\begin{align*}
\limsup_{n \rightarrow \infty}\E_x \Big( \int_{[0, T_{a,n}^+ \wedge \tau_{-M,n}^{-}]} e^{-qs} \diff R_s^{(n)} \Big) \leq \E_x \Big( \limsup_{n \rightarrow \infty} \int_{[0, T_{a,n}^+ \wedge \tau_{-M,n}^{-}]} e^{-qs} \diff R_s^{(n)} \Big), \quad M > 0. \end{align*}

Combining with (1) and Fatou's lemma, we have that 
\begin{multline} \label{bound_fatou}
\mathbb{E}_x\Big( \liminf_{n \rightarrow \infty}\int_{[0,T_{a,n}^{+}]}e^{-qt}\ud R_t^{(n)}\Big)
\leq  \liminf_{n \rightarrow \infty} \mathbb{E}_x\Big( \int_{[0,T_{a,n}^{+}]}e^{-qt}\ud R_t^{(n)}\Big) \\ \leq \limsup_{n \rightarrow \infty} \mathbb{E}_x\Big( \int_{[0,T_{a,n}^{+}]}e^{-qt}\ud R_t^{(n)}\Big) 
\leq \mathbb{E}_x\Big( \limsup_{n \rightarrow \infty}\int_{[0,T_{a,n}^{+}]}e^{-qt}\ud R_t^{(n)}\Big).
\end{multline}
To see how the last inequality holds, by Fatou's lemma, for any $\varepsilon > 0$, we can choose sufficiently large $M$ such that
\begin{align*}
&\limsup_{n \rightarrow \infty} \mathbb{E}_x\Big( \int_{[0,T_{a,n}^{+}]}e^{-qt}\ud R_t^{(n)}\Big) 
\leq \limsup_{n \rightarrow \infty} \mathbb{E}_x\Big( \int_{[0,T_{a,n}^{+} \wedge \tau_{-M,n}^{-} ]}e^{-qt}\ud R_t^{(n)}\Big) + \varepsilon  \\ &\leq \mathbb{E}_x\Big( \limsup_{n \rightarrow \infty}\int_{[0,T_{a,n}^{+} \wedge \tau_{-M,n}^{-} ]}e^{-qt}\ud R_t^{(n)}\Big) + \varepsilon \leq \mathbb{E}_x\Big( \limsup_{n \rightarrow \infty}\int_{[0,T_{a,n}^{+}]}e^{-qt}\ud R_t^{(n)}\Big) + \varepsilon.
\end{align*}

(3) In order to finish the proof, in view of \eqref{bound_fatou}, it remains to show that almost surely 
\begin{align*}
\lim_{n \rightarrow \infty }\int_{[0,T_{a,n}^{+}]}e^{-qt}\ud R_t^{(n)} 
= \int_{[0,T_{a}^{+}]}e^{-qt}\ud R_t. 
\end{align*}
Integration by parts gives 
\begin{align*}
\int_{[0,T_{a,n}^{+}]}e^{-qt}\ud R_t^{(n)} = e^{-q T_{a,n}^+} R_{T_{a,n}^+}^{(n)} + q \int_0^{T_{a,n}^{+}}e^{-qt}R_t^{(n)} \diff t.
\end{align*}
In view of Proposition \ref{prop_approximation}, we have $R_t^{(n)} \xrightarrow{n \uparrow \infty} R_t$, $t \geq 0$, and, as in the proof of Theorem \ref{resol}, $T_{a,n}^+ \xrightarrow{n \uparrow \infty} T_a^+$ a.s.\ (recall that $T_a^+ < \infty$ a.s.\ by Corollary \ref{corollary_one_sided}).

 In addition, the triangle inequality gives
\begin{align*}
\big|e^{-q T_{a,n}^+} R_{T_{a,n}^+}^{(n)} - e^{-q T_{a}^+} R_{T_{a}^+} \big| \leq \big|e^{-q T_{a,n}^+} R_{T_{a,n}^+}^{(n)} - e^{-q T_{a,n}^+} R_{T_{a,n}^+} \big| + \big|e^{-q T_{a,n}^+} R_{T_{a,n}^+} - e^{-q T_{a}^+} R_{T_{a}^+} \big|.
\end{align*}
The first term on the right hand side vanishes by the convergence of $R_t^{(n)}$ and $T_{a,n}^+$. To see how the latter also vanishes, set $\underline{\sigma} := \sup \{ t < T_a^+: V_t =0\}$ (with $\sup \varnothing = 0$) and $\overline{\sigma} := \inf \{ t > T_a^+: V_t =0\}$, then $R_t$ does not increase in the interval $(\underline{\sigma}, \overline{\sigma})$  (clearly $\underline{\sigma} < T_a^+ < \overline{\sigma}$) when $T^+_{a} > 0$; in the case when $T^+_{a} =0$ it does not increase in the interval $[0, \overline{\sigma})$  (clearly $0= T_a^+ < \overline{\sigma}$). This together with the convergence of $T^+_{a,n}$ to $T^+_a$ shows that it indeed vanishes.

\end{proof}

\begin{remark} 
In Proposition \ref{prop_capital_injection},  when $\delta = 0$ or $a=b$, we have $r^{(q)}(x)=Z^{(q)}(x)$ and $\tilde{r}^{(q)}(x)  =\overline{Z}^{(q)}(x) + {\psi'(0+)} / q$, and hence we recover \eqref{capital_injection_identity_SN}.

\end{remark}

By taking $a$ to $\infty$ in identity (\ref{capital cost}), we obtain the following result.
\begin{corollary}
Assume $\psi'(0+) > -\infty$ and $q > 0$.
For any $x \in \R$, we have 
\begin{align*}
\mathbb{E}_x\left(\int_{[0,\infty)} e^{-qt}\ud R_t\right) 
&=-\tilde{r}^{(q)}(x) +\left(\int_b^\infty e^{-\varphi(q) (y-b)}Z^{(q)}(y)\ud y \right)\frac{r^{(q)}(x)} {q \int_b^\infty e^{-\varphi(q) (y-b)}W^{(q)}(y)\ud y}.
\end{align*}
\end{corollary}
\begin{proof}
We have, by \eqref{W_q_limit},
\begin{align*}
\lim_{a \rightarrow \infty} \frac{\tilde{r}^{(q)}(a)} {r^{(q)}(a)} &=\lim_{a \rightarrow \infty}\frac {\int_b^a\mathbb{W}^{(q)}(a-y)Z^{(q)}(y)\ud y }  {q \int_b^a\mathbb{W}^{(q)}(a-y)W^{(q)}(y)\ud y}.
\end{align*}
Here, because $\mathbb{W}^{(q)}(a-y)/\mathbb{W}^{(q)}(a-b) \leq  e^{-\varphi(q)(y-b)}$ by \eqref{W_q_limit}, dominated convergence gives
\begin{align*}
\lim_{a\to\infty}\mathbb{W}^{(q)}(a-b)^{-1}\int_b^a\mathbb{W}^{(q)}(a-y)Z^{(q)}(y)\ud y &=\int_b^{\infty}e^{-\varphi(q)(y-b)}Z^{(q)}(y)\ud y, \\
\lim_{a\to\infty}\mathbb{W}^{(q)}(a-b)^{-1}\int_b^a\mathbb{W}^{(q)}(a-y)W^{(q)}(y)\ud y &=\int_b^{\infty}e^{-\varphi(q)(y-b)}W^{(q)}(y)\ud y.
\end{align*}
This shows the claim.
\end{proof}

\section{Occupation times}  \label{section_occupation_time}
In this section, we are interested in computing the occupation time of the process $V$ above and below the level of refraction $b$. Namely, for $a > 0$, we compute the joint Laplace transform of the stopping time $T_a^+$ and the following quantities: 
\[
\int_0^{T_a^+}1_{\{V_s<b\}}\ud s\qquad\text{and}\qquad \int_0^{T_a^+}1_{\{V_s> b\}}\ud s.
\]
Recall as in Corollary \ref{corollary_one_sided} that $T_a^+ < \infty$ a.s.
\begin{proposition}\label{occupation_time_below}  
For any $p\geq0$, $q\geq -p$, $a > 0$ and $x \leq a$,
\begin{align}
\E_x\left(e^{-pT_a^+-q\int_0^{T_a^+}1_{\{V_s<b\}}\ud s} \right)=\frac{\mathcal{R}^{(p,q)}(x)}{\mathcal{R}^{(p,q)}(a)}, \label{occupation_result_1} \\
\E_x\left(e^{-pT_a^+-q\int_0^{T_a^+}1_{\{V_s> b\}}\ud s} \right)=\frac{\mathcal{L}^{(p,q)}(x)}{\mathcal{L}^{(p,q)}(a)}, \label{occupation_result_2}
\end{align} 
where $\mathcal{R}^{(p,q)}$ is defined as in \eqref{mathcal_R_def} and \begin{align*}
\mathcal{L}^{(p,q)}(x)&:= \mathcal{R}^{(p+q,-q)}(x) \\ &= Z^{(p)}(x)+q\overline{\mathbb{W}}^{(p+q)}(x-b) +p\int_b^x\mathbb{W}^{(p+q)}(x-y)\left(q\overline{W}^{(p)}(y)+\delta W^{(p)}(y)\right)\ud y, \quad x \in \R.
\end{align*}
In particular, for $x \leq 0$, $\mathcal{R}^{(p,q)}(x) = \mathcal{L}^{(p,q)}(x) = 1$.
\end{proposition}
\begin{proof}
(i) We will show the result for the case in which $X$ is of bounded variation; the case of unbounded variation can be obtained by Remark \ref{remark_strongly_approximating} and dominated convergence. 

Let us define, for each $x \leq a$, $h^{(p,q)}(x)$  as the left hand side of  \eqref{occupation_result_1}.
First, consider the case $x<b$. By an application of the strong Markov property, Remark \ref{remark_connection_Y_U}, and \eqref{upcrossing_time_reflected},
\begin{align} \label{p_x_delta_below_b}
h^{(p,q)}(x,a)=\E_x\left[e^{-(p+q)\kappa_b^+}  \right] \E_b\left(e^{-pT_a^+-q\int_0^{T_a^+}1_{\{V_s<b\}}\ud s} \right) 
 =h^{(p,q)}(b,a)\frac{Z^{(p+q)}(x)}{Z^{(p+q)}(b)}.
\end{align}
Second, for the case $x\geq b$, we obtain, again by the strong Markov property and Remark \ref{remark_connection_Y_U}, 
\begin{align} \label{p_x_delta_above_b}
\begin{split}
h^{(p,q)}(x,a)&=\E_x\left(e^{-p\tau_a^+} 1_{\{\tau_a^+< \tau_b^- \}}\right)+\E_x\left[e^{-p\tau_b^-}\E_{Y_{\tau_b^-}}\left(e^{-p T_a^+-q\int_0^{T_a^+}1_{\{V_s<b\}}\ud s} \right) 1_{\{\tau_b^-<\tau_a^+\}}\right]\\
&=\frac{\mathbb{W}^{(p)}(x-b)}{\mathbb{W}^{(p)}(a-b)}+\E_x\left(e^{-p\tau_b^-}h^{(p,q)}(Y_{\tau_b^-},a) 1_{\{\tau_b^-<\tau_a^+ \}}\right).
\end{split}
\end{align}
In order to compute the expectation on the right hand side, 
by \eqref{cc0_general}  and \eqref{p_x_delta_below_b},
\begin{align*} 
\begin{split}
\E_x\left(e^{-p\tau_b^-}h^{(p,q)}(Y_{\tau_b^-},a) 1_{\{\tau_b^-<\tau_a^+ \}}\right) &=\frac{h^{(p,q)}(b,a)}{Z^{(p+q)}(b)} \mathbb{E}_x\left(e^{-p\tau_b^-}Z^{(p+q)}(Y_{\tau_b^-})1_{\{\tau_b^-<\tau_a^+\}}\right) \\
&= \frac{h^{(p,q)}(b,a)}{Z^{(p+q)}(b)} \Big( \mathcal{R}^{(p,q)}(x)-\mathcal{R}^{(p,q)}(a)\frac{\mathbb{W}^{(p)}(x-b)}{\mathbb{W}^{(p)}(a-b)} \Big).
\end{split}
\end{align*}
Substituting this in \eqref{p_x_delta_above_b} gives
\begin{align} \label{p_x_total}
\begin{split}
h^{(p,q)}(x,a)
&=\frac{\mathbb{W}^{(p)}(x-b)}{\mathbb{W}^{(p)}(a-b)}+\frac{h^{(p,q)}(b,a)}{Z^{(p+q)}(b)} \Big( \mathcal{R}^{(p,q)}(x)-\mathcal{R}^{(p,q)}(a)\frac{\mathbb{W}^{(p)}(x-b)}{\mathbb{W}^{(p)}(a-b)} \Big).\end{split}
\end{align}
Setting $x=b$ and solving for $h^{(p,q)}(b,a)$, we have
%
\begin{align*}
h^{(p,q)}(b,a)=\frac{Z^{(p+q)}(b)}{\mathcal{R}^{(p,q)}(a)}. 
\end{align*}
This together with \eqref{p_x_delta_below_b} and \eqref{p_x_total} completes the proof of \eqref{occupation_result_1}.




(ii)  By Lemma \ref{lemma_mass_zero} it is clear that $\int_0^{T_a^+}1_{\{V_s= b\}}\ud s = 0$ a.s. Therefore
\begin{align}
\E_x\left(e^{-pT_a^+-q\int_0^{T_a^+}1_{\{V_s>b\}}\ud s}\right)=\E_x\left(e^{-(p+q)T_a^++q\int_0^{T_a^+}1_{\{V_s< b\}}\ud s} \right), \notag
\end{align} 
 and hence the equation \eqref{occupation_result_2} follows directly from (\ref{occupation_result_1}).
\end{proof}

By setting $\delta = 0$, we can obtain the same identity for the reflected process $U$.
\begin{corollary} 
For any $p\geq0$, $q\geq -p$, $a, b > 0$ and $x \leq a$,
\begin{align*}
\E_x\left(e^{-p\kappa_a^+-q\int_0^{\kappa_a^+}1_{\{U_s<b\}}\ud s} \right)&=\frac{Z^{(p+q)}(x)-q\overline{W}^{(p)}(x-b)-(p+q) q\int_b^xW^{(p)}(x-y)\overline{W}^{(p+q)}(y)\ud y}{Z^{(p+q)}(a)-q\overline{W}^{(p)}(a-b)-(p+q) q\int_b^aW^{(p)}(a-y)\overline{W}^{(p+q)}(y)\ud y},  \\
\E_x\left(e^{-p\kappa_a^+-q\int_0^{\kappa_a^+}1_{\{U_s>  b\}}\ud s} \right)&=\frac{Z^{(p)}(x)+q\overline{W}^{(p+q)}(x-b) +pq\int_b^x W^{(p+q)}(x-y)\overline{W}^{(p)}(y) \ud y}{Z^{(p)}(a)+q\overline{W}^{(p+q)}(a-b) +pq\int_b^a W^{(p+q)}(a-y)\overline{W}^{(p)}(y) \ud y}.
\end{align*} 
\end{corollary}

\appendix

\section{Proofs}

\subsection{Proof of Proposition \ref{prop_approximation}} \label{proof_prop_approximation}
We shall first show the following lemma.
\begin{lemma} \label{lemma_piecewise_bound}

Fix $t > 0$. Let $(x_s; 0 \leq s \leq t )$ and $(\tilde{x}_s; 0 \leq s \leq t )$ be the paths of two different \lev processes such that, for some $\varepsilon > 0$, 
\begin{align*}
\sup_{0 \leq s \leq t} |x_s - \tilde{x}_s| < \varepsilon.
\end{align*}
Fix $z, \tilde{z} \in \R$ and $0 \leq t_0 < t$.

(i)  Define reflected paths $u_s(z, t_0)$ and $\tilde{u}_s(z, t_0)$ on $[t_0, t]$ of the shifted paths $(z + (x_s - x_{t_0});  t_0 \leq s \leq t  )$ and $(\tilde{z} + (\tilde{x}_{s} - \tilde{x}_{t_0});  t_0 \leq s \leq t )$, respectively. In other words, for all $t_0 \leq s \leq t$, let
\begin{align*}
u_s(z, t_0) &:= z + (x_s - x_{t_0}) + ( - \inf_{t_0 \leq u \leq s}  [z + (x_u - x_{t_0})]) \vee 0, \\
\tilde{u}_s(\tilde{z}, t_0) &:= \tilde{z} + (\tilde{x}_s - \tilde{x}_{t_0}) + ( - \inf_{t_0 \leq u \leq s}  [\tilde{z} + (\tilde{x}_u - \tilde{x}_{t_0})]) \vee 0.
\end{align*}
Then, 
\begin{align*}
\sup_{t_0 \leq s \leq t}|u_s(z, t_0) - \tilde{u}_s(\tilde{z}, t_0)| < 2|z - \tilde{z} | + 4 \varepsilon.
\end{align*}

(ii) Similarly, define refracted paths $a_s (z; t_0)$ and $\tilde{a}_s (\tilde{z}; t_0)$ on $[t_0, t]$ that solve, for all $t_0 \leq s \leq t$,
\begin{align*}
a_s (z, t_0) &= z + (x_s - x_{t_0})   - \delta \int_{t_0}^{s} 1_{\{a_u (z, t_0)> b\}} \ud u, \\
\tilde{a}_s (\tilde{z}, t_0) &= \tilde{z} + (\tilde{x}_s - \tilde{x}_{t_0})   - \delta \int_{t_0}^{s} 1_{\{\tilde{a}_u (\tilde{z}, t_0)> b\}} \ud u.
\end{align*}
Then, 
\begin{align*}
\sup_{t_0 \leq s \leq t}|a_s(z, t_0) - \tilde{a}_s(\tilde{z}, t_0)| < 2|z - \tilde{z} | + 4 \varepsilon.
\end{align*}
\end{lemma}
\begin{proof}
(i) Regarding the distance between the shifted paths, the triangle inequality gives
\begin{align} \label{bound_with_tilde}
\sup_{t_0 \leq s \leq t}|[z + (x_{s } - x_{t_0})] - [\tilde{z} + (\tilde{x}_{s } - \tilde{x}_{t_0})]| \leq |z - \tilde{z} | + \sup_{t_0 \leq s \leq t}|x_s - \tilde{x}_s| + |x_{t_0} - \tilde{x}_{t_0}| < |z - \tilde{z} | + 2 \varepsilon.
\end{align}
It can be easily verified that
\begin{multline*}
\Big| \Big( - \inf_{t_0 \leq u \leq s}  [z + (x_u - x_{t_0})] \Big) \vee 0 - \Big( - \inf_{t_0 \leq u \leq s}  [\tilde{z} + (\tilde{x}_u - \tilde{x}_{t_0})] \Big) \vee 0 \Big| \leq
\sup_{t_0 \leq s \leq t}|[z + (x_s - x_{t_0})] - [\tilde{z} + (\tilde{x}_s - \tilde{x}_{t_0})]|.\end{multline*}
Hence  \eqref{bound_with_tilde} and  the triangle inequality gives the result.

(ii) 
As in the proof of Lemma 12 of \cite{KL} (in particular the equation (3.8)),  
\begin{align*}
\sup_{t_0 \leq s \leq t} |a_s(z, t_0) - \tilde{a}_s(\tilde{z}, t_0)| \leq 2 \sup_{t_0 \leq s \leq t}|[z + (x_{s } - x_{t_0})] - [\tilde{z} + (\tilde{x}_{s } - \tilde{x}_{t_0})]|.
\end{align*}
Hence,  \eqref{bound_with_tilde} shows the result.

\end{proof}

We shall now use this lemma to show the proposition.
With $\beta := b/2 > 0$, 
let us define a sequence of increasing random times $(\underline{T}_1, \overline{T}_1, \underline{T}_2, \overline{T}_2, \ldots )$ as follows:
\begin{align*}
&\underline{T}_1 := \inf \{ s > 0: V_s >  \beta\},  \quad \overline{T}_1 := \sup \{ s < \sigma_1: V_s >  \beta\}, \quad \textrm{ with } \sigma_1 := \inf \{ s > \underline{T}_1: V_s =  0\}, 
\end{align*}
and, for all $k \geq 2$, 
\begin{align*}
&\underline{T}_k := \inf \{ s > \sigma_{k-1}: V_s >  \beta \},  \quad \overline{T}_k := \sup \{ s < \sigma_k: V_s >  \beta \},\quad \textrm{ with }  \sigma_k := \inf \{ s > \underline{T}_k: V_s =  0\}.
\end{align*}
For convenience, we also let $\overline{T}_0 := 0$.

Let 
\begin{align*}
K := K_1 + K_2 + 1, \quad \textrm{with } \; K_1 := \sup \{ k \geq 0: \underline{T}_k  < t\} \textrm{ and } K_2 := \sup \{ k \geq 0: \overline{T}_k  < t\},
\end{align*}
be the number of times switching has occurred until time $t$ (plus one), and define
\begin{align}
\underline{\beta} := \min_{1 \leq k \leq K_1}  \inf_{s\in[\underline{T}_k,\overline{T}_k \wedge t)} V_s.\label{beta_underline}\end{align}
Then it is clear from the definitions of $\underline{T}$ and $\overline{T}$ (and $\sigma$) that $\underline{\beta} > 0$.

For the rest of the proof, fix $\omega  \in \Omega \backslash \Omega_0$ where $\Omega_0 := \{ \omega' : \sup_{0 \leq s \leq t} |X_s ( \omega' ) - X_s^{(n)} ( \omega' ) | \nrightarrow 0 \textrm{ or } K ( \omega' )= \infty \}$ is a null set.
It is sufficient to show that there exist a finite number $C$ and $\underline{n}  \in \mathbb{N}$ such that
\begin{align}
\sup_{0 \leq s \leq t} |V^{(n)}_s (\omega) - V_s (\omega)|  \leq C \sup_{0 \leq s \leq t} |X^{(n)}_s (\omega) - X_s (\omega)|, \quad n \geq \underline{n}. \label{bound_with_C}
\end{align}
In this proof, we choose $\underline{n}$
 large enough so that
\begin{align}
\sup_{m \geq \underline{n}}\sup_{0 \leq s \leq t} |X^{(m)}_s (\omega) - X_s (\omega)| < [4(2^{K(\omega)}-1)]^{-1} \underline{\beta} (\omega). \label{bound_sup_sup}
\end{align}
We will see in later arguments that, for all $n \geq \underline{n}$, this bound guarantees that $\underline{T} (\omega)$ and $\overline{T}(\omega)$ can act as switching times for both $V(\omega)$ and $V^{(n)}(\omega)$ (so that, on each interval $[\underline{T}_k(\omega), \overline{T}_k(\omega)]$ and $[\overline{T}_k(\omega), \underline{T}_{k+1}(\omega)]$, both $V(\omega)$ and $V^{(n)}(\omega)$ are refracted and reflected paths, respectively).

Let us fix $n \geq \underline{n}$ and
\begin{align*}
\varepsilon := \sup_{0 \leq s \leq t} |X^{(n)}_s (\omega) - X_s (\omega)|.
\end{align*}
Now we define a sequence $(\eta_k; 0 \leq k \leq K(\omega))$ such that $\eta_0 = 0$ and
\begin{align*}
\eta_{k+1} = 2 \eta_{k} + 4 \varepsilon, \quad k \geq 0.
\end{align*}
The latter gives $\eta_{k+1} - \eta_k = 2 (\eta_k - \eta_{k-1})$, and hence $\eta_{k} - \eta_{k-1} = 2^{k-1} (\eta_1-\eta_0) =  2^{k+1} \varepsilon$.
Therefore, $\eta_k =  (\eta_1-\eta_0) + \cdots + (\eta_k-\eta_{k-1}) = 4(2^k-1) \varepsilon$, and by \eqref{bound_sup_sup}
%
%
\begin{align}
\eta_1 <\eta_2 < \cdots < \eta_{K (\omega)} \leq 4(2^{K(\omega)}-1) \varepsilon < \underline{\beta}(\omega) \leq b/2. \label{bound_eta}
\end{align}
We will show that 
\begin{align} \label{diff_interval_bound}
\begin{split}
\sup_{\overline{T}_k (\omega)\leq s \leq \underline{T}_{k+1} (\omega) \wedge t }
|V_s (\omega)- V^{(n)}_s (\omega)|  &\leq \eta_{2k+1}, \quad 0 \leq k \leq K_2(\omega), \\
\sup_{\underline{T}_k (\omega)\leq s \leq \overline{T}_k (\omega) \wedge t}
|V_s (\omega)- V^{(n)}_s (\omega)|  &\leq \eta_{2k}, \quad 1 \leq k \leq K_1(\omega),
\end{split}
\end{align}
and hence \eqref{bound_with_C} holds with $C = 4(2^{K(\omega)}-1)$.
 Toward this end, we first show the following claims.

\begin{claim}  \label{claim_induction}
(i)  Fix $k \geq 0$. Suppose $\overline{T}_k (\omega) < t$ (or $k \leq K_2(\omega)$) and
\begin{align*}
|(V_{\overline{T}_k-} (\omega) + \Delta X_{\overline{T}_k} (\omega)) - (V_{\overline{T}_k-}^{(n)} (\omega)+  \Delta X_{\overline{T}_k}^{(n)} (\omega)) | &\leq  \eta, \end{align*}
and
\begin{align}
\tilde{\eta} := 2\eta + 4 \varepsilon < \underline{\beta} < b/2. \label{bound_eta_lemma}
\end{align}
Then, 
\begin{align*}
|V_s (\omega)- V^{(n)}_s (\omega)|  \leq \tilde{\eta}, \quad \overline{T}_k (\omega)\leq s \leq \underline{T}_{k+1} (\omega) \wedge t.
\end{align*}

(ii) Fix $k \geq 1$.  Suppose $\underline{T}_k (\omega) < t$ (or $k \leq K_1(\omega)$) and
\begin{align} 
|V_{\underline{T}_k} (\omega)  - V^{(n)}_{\underline{T}_k} (\omega)  | \leq  \eta, \label{v_diff_previous_under}\end{align}
such that \eqref{bound_eta_lemma} holds.
Then, 
\begin{align*}
|V_s  (\omega) - V_s^{(n)}  (\omega) | &\leq \tilde{\eta}, \quad \underline{T}_k(\omega) \leq s < \overline{T}_k(\omega) \wedge t, \\
|(V_{\overline{T}_k-} (\omega) + \Delta X_{\overline{T}_k} (\omega)) - (V_{\overline{T}_k-}^{(n)} (\omega)+  \Delta X_{\overline{T}_k}^{(n)} (\omega)) | &\leq  \tilde{\eta},  \quad  \textrm{if }  \overline{T}_k (\omega) < t. \end{align*}

\end{claim}
\begin{proof}
(i) Consider the reflected paths on $[\overline{T}_k (\omega), \underline{T}_{k+1} (\omega) \wedge t]$: 
\begin{multline} \label{U_time_space_shifted}
U_s (V_{\overline{T}_k-}+ \Delta X_{\overline{T}_k}, \overline{T}_k) (\omega)  := V_{\overline{T}_k-}(\omega) + \Delta X_{\overline{T}_k} (\omega)+ (X_s (\omega) - X_{\overline{T}_k} (\omega)) \\ + \Big( - \inf_{\overline{T}_k  (\omega) \leq u \leq s}  [V_{\overline{T}_k-} (\omega)+ \Delta X_{\overline{T}_k}(\omega)+ (X_u (\omega)- X_{\overline{T}_k} (\omega))] \Big) \vee 0, 
\end{multline}
and 
\begin{multline} \label{U_n_time_space_shifted}
U_s^{(n)} (V_{\overline{T}_k-}^{(n)}+ \Delta X_{\overline{T}_k}^{(n)}, \overline{T}_k) (\omega) := V_{\overline{T}_k-}^{(n)} (\omega)+ \Delta X_{\overline{T}_k}^{(n)}(\omega)+ (X^{(n)}_s (\omega)- X^{(n)}_{\overline{T}_k} (\omega)) \\+ \Big( - \inf_{\overline{T}_k (\omega) \leq u \leq s}  [V_{\overline{T}_k-}^{(n)}(\omega) + \Delta X_{\overline{T}_k}^{(n)}(\omega)+ (X^{(n)}_u(\omega) - X^{(n)}_{\overline{T}_k} (\omega))] \Big) \vee 0.
\end{multline}
By an application of Lemma \ref{lemma_piecewise_bound} (i) with $z = V_{\overline{T}_k-} (\omega)+ \Delta X_{\overline{T}_k}  (\omega)$ and $\tilde{z} = V_{\overline{T}_k-}^{(n)}  (\omega) + \Delta X_{\overline{T}_k}^{(n)}  (\omega)$ and $t_0=\overline{T}_k (\omega)$, we have
\begin{multline*}
\big|U_s (V_{\overline{T}_k-}+ \Delta X_{\overline{T}_k}, \overline{T}_k) (\omega)- U_s^{(n)} (V_{\overline{T}_k-}^{(n)}+ \Delta X_{\overline{T}_k}^{(n)}, \overline{T}_k) (\omega) \big| \\ < 2 \big|[V_{\overline{T}_k-} (\omega)+ \Delta X_{\overline{T}_k}  (\omega) ] - [V_{\overline{T}_k-}^{(n)}  (\omega) + \Delta X_{\overline{T}_k}^{(n)}  (\omega)] \big| + 4 \varepsilon \leq 2 \eta + 4\varepsilon = \tilde{\eta}, \quad \overline{T}_k (\omega) \leq s \leq \underline{T}_{k+1} (\omega) \wedge t.
\end{multline*}

Using that $V_s (\omega) \leq  \beta$ on $[\overline{T}_k (\omega), \underline{T}_{k+1} (\omega)]$ and \eqref{bound_eta_lemma}, we can conclude that there is no refraction for the path $V^{(n)}(\omega)$ on $[\overline{T}_k (\omega), \underline{T}_{k+1} (\omega) \wedge t)$. Therefore $V  (\omega)$ and $V^{(n)}(\omega)$ coincide with their associated reflected paths, defined in \eqref{U_time_space_shifted} and \eqref{U_n_time_space_shifted}, respectively, on $[\overline{T}_k (\omega), \underline{T}_{k+1} (\omega) \wedge t]$, and hence the claim holds.

(ii) Consider the refracted paths $A_s (V_{\underline{T}_k}, \underline{T}_k) (\omega)$ and $A_s^{(n)} (V_{\underline{T}_k}^{(n)}, \underline{T}_k) (\omega)$ on $[\underline{T}_k  (\omega) , \overline{T}_k  (\omega)  \wedge t]$ that solve, for $\underline{T}_k(\omega) \leq s \leq \overline{T}_k  (\omega)  \wedge t$,
\begin{align*}
A_s (V_{\underline{T}_k}, \underline{T}_k) (\omega) &= V_{\underline{T}_k}(\omega) + (X_s (\omega)- X_{\underline{T}_k}(\omega))   - \delta \int_{\underline{T}_k(\omega)}^{s} 1_{\{A_u (V_{\underline{T}_k}, \underline{T}_k)(\omega) > b\}} \ud u,  \\
A_s^{(n)} (V^{(n)}_{\underline{T}_k}, \underline{T}_k) (\omega) &= V_{\underline{T}_k}^{(n)} (\omega) + (X_s^{(n)} (\omega)- X_{\underline{T}_k}^{(n)}(\omega))   - \delta \int_{\underline{T}_k(\omega)}^{s} 1_{\{A_u^{(n)} (V_{\underline{T}_k}^{(n)}, \underline{T}_k)(\omega) > b\}} \ud u.
\end{align*}
  By Lemma \ref{lemma_piecewise_bound} (ii) and \eqref{v_diff_previous_under},
\begin{align*}
|A_s(V_{\underline{T}_k}, \underline{T}_k)  (\omega)  - A^{(n)}_s(V_{\underline{T}_k  }^{(n)}, \underline{T}_k ) (\omega) | < 2|V_{\underline{T}_k} (\omega)  - V_{\underline{T}_k}^{(n)} (\omega) | + 4 \varepsilon \leq 2 \eta + 4\varepsilon = \tilde{\eta}, \quad \underline{T}_k (\omega) \leq s \leq \overline{T}_k (\omega) \wedge t. 
\end{align*}

By the above inequality, \eqref{beta_underline} and \eqref{bound_eta_lemma} ,
there is no reflection for both $V$ and $V^{(n)}$ on $[\underline{T}_k(\omega), \overline{T}_k(\omega) \wedge t)$ and we have $A_s(V_{\underline{T}_k}, \underline{T}_k)  (\omega)  = V_{s-}  (\omega) + \Delta X_s (\omega)$ and $A_s^{(n)}(V_{\underline{T}_k}^{(n)}, \underline{T}_k)  (\omega) = V^{(n)}_{s-} (\omega)  + \Delta X^{(n)}_s (\omega) $  for all $\underline{T}_k (\omega) \leq s \leq \overline{T}_k (\omega) \wedge t$. This completes the proof.
\end{proof}

We are now ready to show \eqref{diff_interval_bound}  by mathematical induction.
The base case is clear by Claim \ref{claim_induction} (i) with $k=0$ and
\begin{align*}
|(V_{0-} (\omega) + \Delta X_0 (\omega)) - (V_{0-}^{(n)} (\omega)+  \Delta X_0^{(n)} (\omega)) | = |x-x| = 0 =\eta_0. \end{align*}
By applying repeatedly Claim \ref{claim_induction} (i) and (ii) one after the other (for $K(\omega)$ times), it is clear that \eqref{diff_interval_bound} holds, as desired.

\subsection{Proof of Lemma  \ref{lemma_useful_identity} (iii).} \label{proof_lemma_useful_identit}

\par Using the fact that $\overline{Z}^{(q)}(x)=x$ for $x\leq0$, we obtain 
\begin{align}
\int_0^{\infty}&\int_{(-\infty,-y)}\left(\overline{Z}^{(q)}(y+u+b-v) - (y+u)\right)\mathbb{W}^{(p)}(x-b-y)\Pi(\ud u)\ud y\notag\\
&=\int_0^{\infty}\int_{(-\infty,-y)}\left(\overline{Z}^{(q)}(y+u+b-v)- \overline{Z}^{(q)}(y+u+0)\right)\mathbb{W}^{(p)}(x-b-y)\Pi(\ud u)\ud y\notag\\
&=\int_v^b \int_0^{\infty}\int_{(-\infty,-y)}Z^{(q)}(y+u+b-z)\mathbb{W}^{(p)}(x-b-y)\Pi(\ud u)\ud y \ud z\notag
\end{align}
which, by \eqref{RLqp2} and \eqref{lemma_useful_identity_c}, equals
\begin{align}\label{Z_bar_y_u_W_integral}
\begin{split}
&\int_v^b \Big[ (c-\delta) {Z^{(q)}(b-z)}\mathbb{W}^{(p)}(x-b)-Z^{(q)}(x-z)-(p-q)\overline{\mathbb{W}}^{(p)}(x-b) \\
&\qquad +\int_b^x\mathbb{W}^{(p)}(x-y)\big((q-p)(Z^{(q)}(y-z)-1)-\delta Z^{(q)\prime}(y-z)\big) \ud y\Big]\ud z\\
&=(c-\delta)\overline{Z}^{(q)}(b-v){\mathbb{W}^{(p)}(x-b)}+\overline{Z}^{(q)}(x-b)-\overline{Z}^{(q)}(x-v) \\
&-\delta\int_b^x\mathbb{W}^{(p)}(x-y)\left(-Z^{(q)}(y-b)+Z^{(q)}(y-v)\right) \ud y\\
&+(q-p)\int_b^x\mathbb{W}^{(p)}(x-y)\left(-\overline{Z}^{(q)}(y-b)+\overline{Z}^{(q)}(y-v)\right) \ud y. \\
&=(c-\delta)\overline{Z}^{(q)}(b-v){\mathbb{W}^{(p)}(x-b)}- \overline{Z}^{(q)}(x-v) + \overline{\mathbb{Z}}^{(p)}(x-b)+\delta\overline{\mathbb{W}}^{(p)}(x-b)\\
&-\delta\int_b^x\mathbb{W}^{(p)}(x-y) Z^{(q)}(y-v)  \ud y +(q-p)\int_b^x\mathbb{W}^{(p)}(x-y) \overline{Z}^{(q)}(y-v) \ud y.
\end{split}
\end{align}
Now we will compute the first term in the right hand side of the previous identity. To this end, an application of 
Lemma 3.1 in \cite{BKY} gives
\begin{align}
\E\left(e^{-p\tau_0^{-}}Y_{\tau_0^-} 1_{\{\tau_0^-<\tau_{x-b}^+ \}}\right)= \frac {\psi'_{Y}(0+)\overline{\mathbb{W}}^{(p)}(x-b)-\overline{\mathbb{Z}}^{(p)}(x-b) }{(c-\delta)\mathbb{W}^{(p)}(x-b)}; \notag
\end{align}
therefore by \eqref{undershoot_expectation} we obtain
\begin{align} \label{y_u_W_integral}
\begin{split}
\int_0^{\infty}\int_{(-\infty,-y)}(y+u)\mathbb{W}^{(p)}(x-b-y)\Pi(\ud u)\ud y&=(c-\delta)\mathbb{W}^{(p)}(x-b)\E\left(e^{-p\tau_0^{-}}Y_{\tau_0^-}1_{\{\tau_0^-<\tau_{x-b}^+\}}\right) \\
&=\psi'_Y(0+)\overline{\mathbb{W}}^{(p)}(x-b)-\overline{\mathbb{Z}}^{(p)}(x-b).
\end{split}
\end{align}
By summing up \eqref{Z_bar_y_u_W_integral} and \eqref{y_u_W_integral} and recalling that $\psi'(0+) = \psi_Y'(0+) + \delta$, we have \eqref{lemma_useful_identity_Z_bar}.


\begin{thebibliography}{99}
\bibitem{APP2007}\sc Avram, F., Palmowski, Z. and Pistorius, M.R. \rm  On the optimal dividend problem for a spectrally negative L\'evy process. {\it Ann. Appl.Probab.} {\bf 17}, 156-180, (2007).

\bibitem{BKY}\sc Bayraktar, E., Kyprianou, A.E., Yamazaki, K. \rm Optimal dividends in the dual model under transaction costs. {\it Insur. Math. Econ.} {\bf 54}, 133-143, (2014).

\bibitem{B} \sc Bertoin, J. {\it L\'evy processes. }\rm Cambridge University Press, Cambridge, (1996).



\bibitem{F1998}
\sc Furrer, H.  \rm Risk processes perturbed by $\alpha$-stable L\'evy motion. {\it Scand. Actuar. J.} 59--74, (1998).


\bibitem{HPSV2004a}\sc Huzak, M., Perman, M., \v{S}iki\'c, H. and Vondra\v cek, Z. \rm Ruin probabilities and decompositions for general perturbed risk  processes.  {\it Ann. Appl. Probab.} {\bf 14}, 1378--1397, (2004).

\bibitem{HPSV2004b} \sc Huzak, M., Perman, M., \v{S}iki\'c, H. and Vondra\v cek, Z. \rm Ruin probabilities for competing claim processes.  {\it J. Appl. Probab.} {\bf 41}, 679--690, (2004).

\bibitem{KKR}\sc  Kuznetsov, A., Kyprianou, A.E., Rivero, V. \rm The
theory of scale functions for spectrally negative L\'evy processes. {\it L\'evy Matters II, Springer Lecture Notes in Mathematics}, (2013).





\bibitem{KKM2004} \sc Kl\"uppelberg, C.,  Kyprianou, A.E. and  Maller, R.A. \rm Ruin probabilities and overshoots for general L\'evy insurance risk  processes.  {\it Ann. Appl. Probab.} {\bf 14}, 1766--1801, (2004).

\bibitem{KK2006} \sc Kl\"uppelberg, C. and  Kyprianou, A.E. \rm On extreme ruinous behaviour of L\'{e}vy insurance risk processes.
{\it J. Appl. Probab.} {\bf 43(2)}, 594--598, (2006).


\bibitem{K} \sc Kyprianou, A.E. {\it Fluctuations of L\'evy processes with applications.} \rm Springer,
Berlin, second edition, (2006).


\bibitem{KL}\sc Kyprianou, A.E., Loeffen, R. \rm Refracted L\'evy
processes. {\it Ann. Inst. H. Poincar\'e,} {46} (1), 24--44, (2010).

\bibitem{KLP} \sc Kyprianou, A.E., Loeffen, R., and P\'erez, J. L. \rm Optimal control with absolutely continuous strategies for spectrally negative L\'evy processes. {\it J. Appl. Probab.} {\bf 49}(1), 150-166, (2012).

\bibitem{KP2007} \sc Kyprianou, A.E. and Palmowski, Z. \rm Distributional study of de Finetti's dividend problem for a general L\'evy insurance risk process. {\it J. Appl. Probab.} {\bf 44}, 349-365, (2007)

\bibitem{KPP} \sc Kyprianou, A.E., Pardo, J. C., and P\'erez, J. L. \rm Occupation times of refracted L\'evy processes. {\it  J. Theoret. Probab.} {\bf 27}(4), 1292-1315, (2014).










\bibitem{Loeffen}  Loeffen, R.L., On optimality of the barrier strategy in de Finetti's dividend problem for spectrally negative L\'evy processes. {\it Ann. Appl.Probab.} {\bf 18} (5),1669-1680, (2008).

\bibitem{LoRZ} Loeffen, R.L., Renaud, J-F. and Zhou, X. 
Occupation times of intervals until first passage times for spectrally negative L\'evy processes with applications.  {\it  Stochastic Process. Appl.}, {\bf 124 (3)}, 1408--1435, (2014).

\bibitem{PY_astin} P\'erez, J.L. and Yamazaki, K. 
Refraction-reflection strategies in the dual model.  {\it  ASTIN Bulletin}, (forthcoming).
 
 \bibitem{Pistorius_2003} Pistorius, M.R. \rm On doubly reflected completely asymmetric {L}\'evy processes. 
{\it Stochastic Process. Appl.}, {\bf 107 (1)}, 131--143, (2003).

\bibitem{P2004}\sc Pistorius, M.R. \rm  On exit and ergodicity of the spectrally one-sided L\'evy process reflected at its infimum. {\it J. Theoret. Probab.} {\bf 17} (1), 183--220, (2004).

\bibitem{P2007}\sc Pistorius, M.R. \rm  An excursion-theoretical approach to some boundary crossing problems and the Skorokhod embedding for reflected L\'evy processes. {\it Seminaire de Probabilit\'es XL}, 287--307, (2007).


\bibitem{R2014} \sc Renaud, J-F. \rm On the time spent in the red by a refracted L\'evy risk process. {\it J. Appl. Probab.} {\bf 51} (4),  1171-1188, (2014).

\bibitem{RZ2007} \sc Renaud, J-F. and Zhou, X. \rm  Distribution of the dividend payments in a general L\'evy risk model.  {\it  J. Appl. Probab.} {\bf 44},  420--427, (2007).

\bibitem{SV2007} \sc Song, R. and  Vondra\v{c}ek, Z. \rm  On suprema of L\'evy processes and application in risk theory. {\it Ann.  lnst. H. Poincar\'e}, {\bf 44}, 977--986, (2008).



\end{thebibliography}
\end{document}